\documentclass{amsart}
\usepackage{tikz}
\usepackage{xcolor}
\usepackage{amssymb,latexsym,amsmath,extarrows}
\usepackage{graphicx,mathrsfs,comment}
\usepackage{hyperref,url}
\usepackage{mathtools}
\usepackage{mathabx}

\usepackage{amstext}
\usepackage{bbm}

\numberwithin{equation}{section}

\renewcommand{\theequation}{\arabic{section}.\arabic{equation}}

\newtheorem{theorem}{Theorem}[section]
\newtheorem{lemma}[theorem]{Lemma}

\newtheorem{proposition}[theorem]{Proposition}

\newtheorem{corollary}[theorem]{Corollary}

\newcommand{\bbk}{{\mathbb{K}}}

\begin{document}
\hypersetup{hidelinks}

\title[Positive measure for restricted projections]{Positive measure for restricted projections along smooth non-degenerate curves}

\author{Seonghun Jeon} \address{ Seonghun Jeon \\ Department of Mathematical Sciences, Seoul National University, Republic of Korea}\email{baba4409@snu.ac.kr}

\author{Changkeun Oh}\address{ Changkeun Oh\\ Department of Mathematical Sciences and RIM, Seoul National University, Republic of Korea} \email{changkeun@snu.ac.kr}

\begin{abstract}
Let $1\leq m<n$, and let $\gamma:[0,1]\to\mathbb R^n$ be a smooth
non-degenerate curve.  For $\theta\in[0,1]$, set
\[
    V_\theta
    :=
    \operatorname{span}
    \{\gamma'(\theta),\ldots,\gamma^{(m)}(\theta)\},
\]
and let $\pi_\theta$ denote the orthogonal projection onto $V_\theta$. 
We prove that if $E\subset\mathbb R^n$ is a Borel set with
$\dim_{\mathrm H}E>m$, then
\[
    \mathcal H^m(\pi_\theta(E))>0
\]
for almost every $\theta\in[0,1]$.  
\end{abstract}

\maketitle

\section{Introduction}

Fix integers $1\leq m<n$.  Let $\gamma:[0,1]\to\mathbb R^n$ be a smooth curve satisfying
\begin{equation}\label{eq:nondegenerate}
    \det\big[\gamma'(\theta),\ldots,\gamma^{(n)}(\theta)\big]\neq0,
    \qquad \theta\in[0,1].
\end{equation}
Such a curve will be called non-degenerate.  A model example is the moment curve
\begin{equation}\label{eq:model-moment-curve}
    \gamma_\circ(\theta)=\left(\theta,\frac{\theta^2}{2!},\ldots,\frac{\theta^n}{n!}\right).
\end{equation}
For $\theta\in[0,1]$, let
\begin{equation}\label{eq:Vtheta}
    V_\theta:=\operatorname{span}\{\gamma'(\theta),\ldots,\gamma^{(m)}(\theta)\},
\end{equation}
and let $\pi_\theta:\mathbb R^n\to V_\theta$ denote the orthogonal projection. The parameter $\theta$ indexes the family of subspaces
$\{V_\theta\}_{\theta\in[0,1]}$ and the associated family of projections
$\{\pi_\theta\}_{\theta\in[0,1]}$.

\begin{theorem}\label{thm:main}
Let $E\subset B^n(0,1)$ be Borel measurable and suppose that
\[
    \dim_{\mathrm H}E>m.
\]
Then
\[
    \mathcal H^m(\pi_\theta(E))>0
\]
for almost every $\theta\in[0,1]$.
\end{theorem}

Theorem~\ref{thm:main} easily follows from the following measure-theoretic statement.

\begin{theorem}\label{thm:measure}
Let $m<\alpha<n$. Then there exists $1<q<2$ such that, for every
nonzero finite Borel measure $\mu$ supported on $B^n(0,1)$ satisfying
\begin{equation}\label{eq:frostman}
    c_\alpha(\mu)
    :=\sup_{x\in\mathbb R^n,\,r>0}
    \frac{\mu(B(x,r))}{r^\alpha}
    \leq1,
\end{equation}
we have
\begin{equation}\label{eq:averaged-Lq}
    \int_0^1
    \|\pi_{\theta\#}\mu\|_{L^q(V_\theta)}^q\,d\theta
    \lesssim_\alpha
    \mu(\mathbb R^n).
\end{equation}
In particular, $\pi_{\theta\#}\mu$ is absolutely continuous with
respect to $\mathcal H^m|_{V_\theta}$ for a.e.
$\theta\in[0,1]$.
\end{theorem}
\begin{proof}[Proof that Theorem~\ref{thm:measure} implies Theorem~\ref{thm:main}]
Let
\[
m<\alpha<\dim_{\mathrm H}E.
\]
By Frostman's lemma, there exists a nonzero finite Borel measure \(\mu\) supported on \(E\) such that
\[
\mu(B(x,r))\lesssim r^\alpha,
\qquad x\in\mathbb R^n,\quad r>0.
\]
After multiplying $\mu$ by a positive constant, we may assume that
$c_\alpha(\mu)\leq1$. Hence Theorem~\ref{thm:measure} applies. Consequently,
\[
\pi_{\theta\#}\mu
   \ll \mathcal H^m|_{V_\theta}
\]
for almost every \(\theta\in[0,1]\).

Fix such a \(\theta\). Since \(\mu\) is supported on \(E\),
\[
(\pi_{\theta\#}\mu)(\pi_\theta(E))
 =\mu\bigl(\pi_\theta^{-1}(\pi_\theta(E))\bigr)
 =\mu(\mathbb R^n)>0.
\]
On the other hand, if \(\mathcal H^m(\pi_\theta(E))=0\), then the absolute continuity of \(\pi_{\theta\#}\mu\) would give
\[
(\pi_{\theta\#}\mu)(\pi_\theta(E))=0,
\]
a contradiction. Therefore
\[
\mathcal H^m(\pi_\theta(E))>0
\]
for almost every \(\theta\in[0,1]\), as required.
\end{proof}

Combining Theorem~\ref{thm:measure} with
\cite[Theorem~1.2]{Mattila2024}, we obtain the following.

\begin{corollary}\label{cor:sections}
Let $m<s<n$, and let $E\subset B^n(0,1)$ be Borel measurable with
$0<\mathcal H^s(E)<\infty$. Then, for almost every
$\theta\in[0,1]$,
\[
\mathcal H^m\Bigl(
    \bigl\{u\in V_\theta:
    \dim_{\mathrm H}
    \bigl(E\cap\pi_\theta^{-1}\{u\}\bigr)=s-m
    \bigr\}
\Bigr)>0.
\]
\end{corollary}

Theorem \ref{thm:main} extends the previously known cases  $(n,m)=(2,1), (3,2)$ and $(3,1)$ to all $1 \leq m<n$.
The first case follows from Marstrand’s projection theorem \cite{Marstrand1954}, while the latter two were proved in \cite[Theorem~2]{gan2022restricted} and \cite{MR4694174}, respectively. Gan, Guo, and Wang \cite{gan2024restricted} proved that
\begin{equation}
    \dim_H \pi_{\theta}(E) = \min \{m, \dim_H E\}
\end{equation}
for almost every $\theta$. Theorem \ref{thm:main} establishes the corresponding positive-measure result when $\dim_H E >m$.
\medskip

The proof combines the good--bad wave packet decomposition of
\cite{gan2022restricted}, originated in \cite{GuthIosevichOuWang2020}, with the multiscale frequency decomposition and iterated
decoupling scheme of \cite[Sections~3--5]{gan2024restricted}. Let us explain the main difficulties in extending the three-dimensional arguments to general $n$ and $m$. On the one hand, in the three dimensional case, the frequency is decomposed according to the distance to ``the most degenerate cone''. In higher dimensions, ``the most degenerate cone'' is not a hypersurface, and it is not entirely clear how to decompose the frequencies. On the other hand, even after obtaining such a frequency decomposition, following the strategy of \cite{gan2022restricted} would require a  refined decoupling for the resulting frequency pieces. However, the resulting frequency pieces involve multiscales, and it is not clear how to obtain a refined decoupling for the pieces. To overcome these difficulties, we introduce the frequency decomposition of $\mu$ depending on the parameter $\theta$. We also avoid the use of refined decouplings. In particular, our proof slightly simplifies the proofs of \cite[Theorem~2]{gan2022restricted} and \cite{MR4694174} in the sense that we do not use refined decouplings.

\subsection*{Notation and parameters}

Implicit constants may depend on $n,m$, on the fixed curve $\gamma$, and
on fixed cutoff functions. Unless otherwise indicated, an unsubscripted
constant $C>0$ has only this permitted dependence and may change from
line to line. In particular, all harmless losses arising from the
$2^{\delta j}$-enlargement of the wave packets are denoted by
$2^{C\delta j}$; different occurrences need not involve the same
numerical value of $C$, and finite products of such losses are absorbed
into a single factor of the same form. We use $A\lesssim B$ to mean
$A\leq CB$. A fixed dilation of a box is harmless and may change from
line to line. The notation $\operatorname{RapDec}(2^j)$ denotes a
quantity bounded by $C_N2^{-Nj}$ for every $N$.

We use only two small parameters. We shall choose $0<\epsilon\ll1$ and
set
\begin{equation}\label{eq:alpha-zero}
    \alpha_0:=m-\epsilon.
\end{equation}
We retain the notation $\alpha_0$ throughout the argument and substitute
$m-\alpha_0=\epsilon$ only in the final exponent calculation. Whenever
a decoupling or pigeonholing argument permits an arbitrarily small
$2^{\eta j}$ loss, we invoke it with $\eta$ equal to a sufficiently
small fixed multiple of $\epsilon$; all such losses are absorbed into
factors $2^{C\epsilon j}$. The parameter $\delta>0$ enlarges the wave
packets and is chosen last, with $\delta\ll\epsilon$.

Throughout the proof, $\mu$ satisfies the normalization
\begin{equation}\label{eq:normalization}
    c_\alpha(\mu)\leq1,
    \qquad 0<\mu(\mathbb R^n)\leq1.
\end{equation}
The second inequality follows from $\operatorname{supp}\mu\subset
B^n(0,1)$ and the first.

\subsection{Acknowledgements}

Changkeun Oh was supported by the POSCO Science Fellowship of POSCO TJ Park Foundation, and the National
Research Foundation of Korea (NRF) grant funded by the Korea government (MSIT) RS-2024-00341891. The authors are grateful to Shengwen Gan for sharing his note on restricted orthogonal projections onto lines, and for valuable discussions. The authors also would like to thank Terry Harris for introducing Corollary \ref{cor:sections} to us.

\subsection{Tool and computational resource disclosure}

ChatGPT-5.5 Pro was used solely to assist with language editing and the presentation of the manuscript. It was not used to generate any mathematical results or proofs. All mathematical content was independently verified by the authors, who take full responsibility for the manuscript.

\section{Preliminaries}\label{sec2}

In this section we set up the frequency geometry used throughout the proof.
In Subsection 2.1, we introduce the derivative frame and its orthonormal
Frenet frame, compare boxes expressed in these two frames, and state the precise frequency decomposition of the
annular portion of $V_\theta$. The proof of
this decomposition, obtained by following the iterative construction
of \cite[Sections~3--5]{gan2024restricted}, is given in the appendix. In Subsection~2.2, we discretize the parameter $\theta$.

\subsection{Frames and the $V_\theta$-adapted frequency decomposition}
Set
\[
    v_r(\theta):=\gamma^{(r)}(\theta),
    \qquad 1\leq r\leq n.
\]
By the non-degeneracy assumption \eqref{eq:nondegenerate}, the ordered
$n$-tuple
\[
    \bigl(v_1(\theta),\ldots,v_n(\theta)\bigr)
\]
is a basis of $\mathbb R^n$ for every $\theta\in[0,1]$. We call this
$\theta$-dependent ordered basis the \emph{derivative frame} of
$\gamma$. Let
\[
    e_1(\theta),\ldots,e_n(\theta)
\]
be its Gram--Schmidt orthonormalization, which we call the
\emph{Frenet frame}.

For a nonincreasing vector of side lengths
\[
    \mathbf L=(L_1,\ldots,L_n),
    \qquad L_1\geq\cdots\geq L_n>0,
\]
define
\begin{align}
B_v(\theta;\mathbf L)
&:=
\left\{
    \sum_{r=1}^n\lambda_rv_r(\theta):
    |\lambda_r|\leq L_r
\right\},
\label{eq:derivative-box}\\
B_e(\theta;\mathbf L)
&:=
\left\{
    \sum_{r=1}^n\lambda'_re_r(\theta):
    |\lambda'_r|\leq L_r
\right\}.
\label{eq:frenet-box}
\end{align}

We shall sometimes replace some of the coordinate bounds
\[
    |\lambda_r|\leq L_r
\]
by the two-sided bounds
\[
    |\lambda_r|\sim L_r,
    \qquad r\in\mathcal A,
\]
for some $\mathcal A\subset\{1,\ldots,n\}$. A \emph{sign component}
of such a region is obtained by fixing signs
$\sigma_r\in\{-1,1\}$ for $r\in\mathcal A$ and replacing the preceding
conditions by
\[
    C^{-1}L_r\leq \sigma_r\lambda_r\leq CL_r,
    \qquad r\in\mathcal A,
\]
where $C>0$ is the constant implicit in $\sim$. Thus each sign
component is a rectangular box, and the original region is the union
of at most $2^{|\mathcal A|}$ such boxes.

\begin{lemma}[Derivative-frame and Frenet-frame boxes]
\label{lem:frame-comparison}
There is a constant $C=C_\gamma\geq1$ such that, uniformly for
$\theta\in[0,1]$ and every nonincreasing $\mathbf L$,
\begin{equation}\label{eq:frame-comparison}
    C^{-1}B_e(\theta;\mathbf L)
    \subset B_v(\theta;\mathbf L)
    \subset CB_e(\theta;\mathbf L).
\end{equation}
More generally, let $\mathcal A\subset\{1,\ldots,n-1\}$ and suppose
that
\begin{equation}\label{eq:active-gap}
    L_{r+1}\leq c_\gamma L_r,
    \qquad r\in\mathcal A,
\end{equation}
where $c_\gamma>0$ is sufficiently small in terms of the fixed curve.
Then every sign component of a derivative-frame region satisfying
\[
    |\lambda_r|\sim L_r,
    \qquad r\in\mathcal A,
\]
is contained in a fixed dilation of an $O_n(1)$-union of
Frenet-frame sign components satisfying
\[
    |\lambda'_r|\sim L_r,
    \qquad r\in\mathcal A,
\]
and conversely.
\end{lemma}

\begin{proof}
Write
\[
    V(\theta)
    =
    [v_1(\theta)\ \cdots\ v_n(\theta)],
    \qquad
    E(\theta)
    =
    [e_1(\theta)\ \cdots\ e_n(\theta)].
\]
Gram--Schmidt gives
\begin{equation}\label{eq:QR-frame}
    V(\theta)=E(\theta)R(\theta),
\end{equation}
where $R(\theta)$ is upper triangular. By smoothness, compactness,
and \eqref{eq:nondegenerate}, the matrices $R(\theta)$ and
$R(\theta)^{-1}$ are uniformly bounded, and their diagonal entries
are uniformly bounded away from zero.

If
\[
    x=V(\theta)\lambda=E(\theta)\lambda',
\]
then $\lambda'=R(\theta)\lambda$. Hence
\[
    |\lambda'_r|
    \leq
    \sum_{q=r}^n|R_{rq}(\theta)|L_q
    \lesssim_\gamma L_r,
\]
because $L_q\leq L_r$ for $q\geq r$. This proves the second inclusion
in \eqref{eq:frame-comparison}. The first follows in the same way
from the upper triangular matrix $R(\theta)^{-1}$.

For the sign-component assertion, if $r\in\mathcal A$, then
\[
    \lambda'_r
    =
    R_{rr}(\theta)\lambda_r
    +
    \sum_{q>r}R_{rq}(\theta)\lambda_q.
\]
By \eqref{eq:active-gap}, the second term has size at most
$C_\gamma c_\gamma L_r$. Taking $c_\gamma$ sufficiently small
compared with the uniform lower bound for $|R_{rr}|$ preserves the
lower bound
\[
    |\lambda'_r|\gtrsim L_r.
\]
The corresponding upper bound follows from
\eqref{eq:frame-comparison}. Splitting according to the finitely many
choices of signs gives the stated $O_n(1)$ collection. Applying the
same argument to $R^{-1}$ proves the converse.
\end{proof}

The frequency decomposition of \cite{gan2024restricted} is naturally expressed
in the derivative frame. Lemma~\ref{lem:frame-comparison} shows that,
after an $O_n(1)$ subdivision into sign components, every resulting
piece is uniformly comparable with a piece having the same side
lengths in the Frenet frame. We therefore formulate all frequency
regions below directly in Frenet coordinates.

For fixed $\theta\in[0,1]$ and $j\geq1$, define
\begin{equation}\label{eq:Ptheta-j}
\begin{split}
P_{\theta,j}
:=
\Bigl\{
    \sum_{r=1}^n\lambda_re_r(\theta):
    &\ |\lambda_r|\lesssim2^j
       \quad (1\leq r\leq m),\\
    &\ |\lambda_r|\lesssim1
       \quad (m<r\leq n),\\
    &\ \max_{1\leq r\leq m}|\lambda_r|
       \sim2^j
\Bigr\}.
\end{split}
\end{equation}
Since the frame is orthonormal, there is a fixed constant $C$ such
that
\begin{equation}\label{eq:V-covered-by-P}
    V_\theta
    \subset
    B(0,C)\cup\bigcup_{j\geq1}P_{\theta,j}.
\end{equation}
Choose smooth cutoffs $\Phi_{\theta,j}$ adapted to
$P_{\theta,j}$, supported in $2P_{\theta,j}$, such that
\[
    \sum_{j\geq1}\Phi_{\theta,j}(\xi)=1,
    \qquad
    \xi\in V_\theta\setminus B(0,C).
\]

The precise output of the decomposition needed in this paper is the
following.

\begin{proposition}[$V_\theta$-adapted frequency decomposition]
\label{prop:Vtheta-decomposition}
Fix $\theta\in[0,1]$ and $j\geq1$. There is a finite collection of
output data of the following form:
\begin{equation}
     \bigl(J,(m_i)_{i=1}^J,(n_i)_{i=1}^J,\bbk\bigr)
\end{equation}
with
\[
    1\leq J\lesssim_n1,
\]
\[
    1\leq m_1\leq\cdots\leq m_J\leq m,
    \qquad
    n=n_1\geq\cdots\geq n_J=m+1,
\]
and
\[
    \bbk=(k_1,\ldots,k_J),
    \qquad k_i\in\mathbb N.
\]
For a fixed output, let
\[
    \mathcal M_{\bbk}
    :=
    \{m_1,\ldots,m_J\}
    \subset\{1,\ldots,m\},
\]
and, for $1\leq r\leq m$, define
\begin{equation}\label{eq:a-r}
    a_r=a_r(j,\bbk)
    :=
    j-\sum_{i: m_i < r}(r-m_i)k_i.
\end{equation}
Then $P_{\theta,j}$ admits a smooth decomposition subordinate to the regions
\begin{equation}\label{eq:Omega-compact}
\begin{split}
\Omega^j_{\bbk,\theta}
:=
\Bigl\{
    \sum_{r=1}^n\lambda_re_r(\theta):
    &\ |\lambda_r|\lesssim2^{a_r}
       \quad
       \bigl(
          r\in\{1,\ldots,m\}\setminus\mathcal M_{\bbk}
       \bigr),\\
    &\ |\lambda_r|\sim2^{a_r}
       \quad
       \bigl(
          r\in\mathcal M_{\bbk}
       \bigr),\\
    &\ |\lambda_r|\lesssim1
       \quad
       (m<r\leq n)
\Bigr\}.
\end{split}
\end{equation}
For every output and every \(q=m+1,\ldots,n\), the parameters satisfy
\begin{equation}\label{eq:output-identities}
    0
    \leq
    j-\sum_{i:n_i\geq q}(q-m_i)k_i
    < n.
\end{equation}
In particular,
\[
    \sum_{i:n_i\geq q}(q-m_i)k_i
    =
    j+O_n(1).
\]
\end{proposition}

\begin{proof}
See Appendix~\ref{app:frequency-decomposition}. There we follow the
iterative decomposition in \cite[Sections~3--5]{gan2024restricted}, keep track
of the active indices and the anisotropic rescalings, and then use
Lemma~\ref{lem:frame-comparison} to pass from the derivative frame to
the Frenet frame.
\end{proof}

Here and below, an \emph{output} means the full collection of data
\[
    \bigl(J,(m_i)_{i=1}^J,(n_i)_{i=1}^J,\bbk\bigr).
\]
For brevity, we label an output by $\bbk$ and leave the corresponding
sequences $(m_i)$ and $(n_i)$ implicit. 

Each region $\Omega^j_{\bbk,\theta}$ is the union of $O_n(1)$
rectangular sign components. We call these rectangular components
the \emph{output boxes}. All Fourier support statements below are
understood up to fixed dilations of the corresponding output boxes.

\subsection{Discretization in $\theta$}

For an output $\bbk$, define
\begin{equation}\label{eq:K-ell}
    K_{\bbk}
    :=
    \sum_{i=1}^J k_i,
    \qquad
    \ell_{\bbk}
    :=
    2^{-K_{\bbk}}.
\end{equation}
To discretize the parameter $\theta$, we use the following lemma. The lemma says that the sets $\Omega_{\bbk,\theta}^j$ and $\Omega_{\bbk,\theta'}^j$ are morally the same whenever $|\theta-\theta'| \lesssim \ell_{\bbk}$.

\begin{lemma}[Stability of output boxes under nearby parameters]
\label{lem:nearby-output-boxes}
There exists a constant $c_0=c_0(\gamma)>0$ with the following
property. Fix an output $\bbk$, and suppose that
\[
    |\theta-\theta'|\leq c_0\ell_{\bbk}.
\]
For a fixed choice of signs in the coordinates indexed by
$\mathcal M_{\bbk}$, the corresponding output boxes of
$\Omega^j_{\bbk,\theta}$ and $\Omega^j_{\bbk,\theta'}$ are contained
in fixed dilations of one another.
\end{lemma}

\begin{proof}
The Frenet frame satisfies
\[
    e_r'(t)\in
    \operatorname{span}\{e_{r-1}(t),e_{r+1}(t)\},
\]
with the usual convention at the endpoints. 
Consequently,
\[
    e_r^{(k)}(t)
    \in
    \operatorname{span}
    \{e_s(t):|s-r|\leq k\}.
\]

It follows by Taylor expansion that, for $r\neq r'$,
\begin{equation}\label{eq:frenet-variation}
    \bigl|
        \langle e_r(\theta),e_{r'}(\theta')\rangle
    \bigr|
    \lesssim_\gamma
    |\theta-\theta'|^{|r-r'|}.
\end{equation}
Indeed, if $d=|r-r'|$, then all terms of order less than $d$ in the
Taylor expansion of $e_r(\theta)$ at $\theta'$ are orthogonal to
$e_{r'}(\theta')$. We also have
\begin{equation}\label{eq:frenet-diagonal}
    \langle e_r(\theta),e_r(\theta')\rangle
    =1+O_\gamma(|\theta-\theta'|^2).
\end{equation}

Set
\[
    L_r:=
    \begin{cases}
        2^{a_r}, & 1\leq r\leq m,\\
        1,       & m<r\leq n.
    \end{cases}
\]
The definition of $a_r$, together with
\eqref{eq:output-identities} for $q=m+1$, gives
\[
    L_r\lesssim 2^{(r'-r)K_{\bbk}}L_{r'},
    \qquad r<r'.
\]
Writing $\Delta=|\theta-\theta'|$ and using
$\Delta\leq c_0 2^{-K_{\bbk}}$, we therefore obtain
\begin{equation}\label{eq:nearby-frame-mixing}
    L_r\Delta^{|r-r'|}
    \lesssim c_0^{|r-r'|}L_{r'},
    \qquad r\neq r'.
\end{equation}

Let
\[
    x=\sum_{r=1}^n\lambda_re_r(\theta)
\]
belong to an output box of $\Omega^j_{\bbk,\theta}$, and write
\[
    x=\sum_{r'=1}^n\lambda'_{r'}e_{r'}(\theta').
\]
Then
\[
    \lambda'_{r'}
    =
    \sum_{r=1}^n
    \lambda_r
    \langle e_r(\theta),e_{r'}(\theta')\rangle.
\]
Since $|\lambda_r|\lesssim L_r$, equations
\eqref{eq:frenet-variation} and
\eqref{eq:nearby-frame-mixing} imply
\[
    |\lambda'_{r'}|\lesssim L_{r'}
\]
for every $r'$.

If $r'\in\mathcal M_{\bbk}$, then $|\lambda_{r'}|\sim L_{r'}$, and
\eqref{eq:frenet-diagonal} gives
\[
    \lambda'_{r'}
    =
    \lambda_{r'}+O_\gamma(c_0L_{r'}).
\]
Choosing $c_0$ sufficiently small preserves both the sign of
$\lambda_{r'}$ and the condition $|\lambda'_{r'}|\sim L_{r'}$. Thus $x$
belongs to a fixed dilation of the corresponding output box of
$\Omega^j_{\bbk,\theta'}$. Interchanging $\theta$ and $\theta'$
proves the converse.
\end{proof}

Lemma~\ref{lem:nearby-output-boxes} allows us to replace the
continuously $\theta$-indexed output boxes by boxes based at a fixed
discrete set of parameters.

Fix the constant $c_0$ from
Lemma~\ref{lem:nearby-output-boxes}. For each output $\bbk$, let
$\Theta_{\bbk}$ be a maximal $c_0\ell_{\bbk}$-separated subset of
$[0,1]$. Since $c_0$ is fixed, we regard $\Theta_{\bbk}$ as an
$\ell_{\bbk}$-net. Given $\theta\in[0,1]$, choose
\[
    \theta_{\bbk} \in\Theta_{\bbk}
\]
such that
\begin{equation}\label{eq:nearest-output-point}
    |\theta-\theta_{\bbk}|
    \leq c_0\ell_{\bbk}.
\end{equation}
We divide [0,1] into intervals of length comparable to $\ell_{\bbk}$, so that the chosen $\theta_{\bbk}$ is fixed on each interval. We call these the \emph{output intervals}; each has length
comparable to $\ell_{\bbk}$.
By Lemma~\ref{lem:nearby-output-boxes}, for each output $\bbk$,
every output box of $\Omega^j_{\bbk,\theta}$ is contained in a fixed
dilation of the corresponding output box of
$\Omega^j_{\bbk,\theta_{\bbk}}$, and conversely. Thus, replacing the
boxes separately for each $\bbk$ and then taking the union over all
outputs, we obtain a net-indexed family that still covers the support
of $\Phi_{\theta,j}$ on $V_\theta$.

For fixed $\theta$, let $\mathcal T_{j,\bbk}(\theta)$ be the
collection of rectangular output boxes of
$\Omega^j_{\bbk,\theta_{\bbk}}$. Choose smooth cutoffs
\[
    \psi_{\tau,\theta},
    \qquad
    \tau\in\mathcal T_{j,\bbk}(\theta),
\]
uniformly adapted to $\tau$, supported in $2\tau$, and satisfying
\begin{equation}\label{eq:net-indexed-partition}
    \Phi_{\theta,j}
    =
    \sum_{\bbk}
    \sum_{\tau\in\mathcal T_{j,\bbk}(\theta)}
    \psi_{\tau,\theta}
    \qquad\text{on }V_\theta.
\end{equation}
Here the supporting box $\tau$ is based at the discrete parameter
$\theta_{\bbk}$, whereas the cutoff $\psi_{\tau,\theta}$ is allowed to
depend on the actual parameter $\theta$, since the partition is required
to hold on $V_\theta$ for $\Phi_{\theta,j}$.

The supporting boxes are indexed by the fixed nets
$\Theta_{\bbk}$. For each fixed $j$ and $\bbk$, the resulting
net-indexed family of boxes has uniformly bounded overlap.

\section{Wave packets and the good--bad decomposition}

In this section we use the net-indexed output boxes of Section~2 to
construct a wave packet decomposition.  We decompose each piece of \(\mu\) localized to an output box
\(\tau\) in frequency space into spatial packets adapted to translates
of the corresponding dual plank \(\tau^*\).  The Fourier slice theorem gives the corresponding decomposition of
\(\pi_{\theta\#}\mu\), and the localization of the packets allows us to
separate the contribution from the dyadic scales \(j\geq1\) into bad
and good components.
Sections~4 and~5 treat these two components, respectively.

\subsection{Wave packet decomposition}

For $\tau \in \mathcal T_{j,\bbk}(\theta)$, let $\mathbb T_\tau$ be a finitely overlapping cover of $\mathbb R^n$
by translates of $2^{\delta j}\tau^*$.  Choose a smooth partition of
unity $\{\eta_T:T\in\mathbb T_\tau\}$ adapted to this covering, with
\begin{equation}\label{eq:eta-partition}
    \sum_{T\in\mathbb T_\tau}\eta_T=1,
    \qquad
    \operatorname{supp}\widehat{\eta_T}
    \subset\tfrac1{100}(\tau-\tau).
\end{equation}
For $T\in\mathbb T_\tau$, define
\begin{equation}\label{eq:M-T-theta}
    M_{T,\theta}\mu
    :=\eta_T(\mu*\check\psi_{\tau,\theta}).
\end{equation}
Then
$\operatorname{supp}\widehat{M_{T,\theta}\mu}\subset2\tau$.
\begin{lemma}[Fourier slice theorem]\label{lem:fourier-slice}
Let $\nu$ be a finite complex Borel measure on $\mathbb R^n$.  Then, for every
$\theta\in[0,1]$,
\[
    \widehat{\pi_{\theta\#}\nu}(\xi)
    =
    \widehat{\nu}(\xi),
    \qquad \xi\in V_\theta,
\]
where the Fourier transform on the left is taken on the subspace
$V_\theta$.
\end{lemma}

\begin{proof}
For $\xi\in V_\theta$, the fact that $\pi_\theta$ is the orthogonal
projection onto $V_\theta$ gives
\[
    \langle \pi_\theta x,\xi\rangle
    =
    \langle x,\xi\rangle.
\]
Therefore
\[
\begin{split}
    \widehat{\pi_{\theta\#}\nu}(\xi)
    &=\int_{V_\theta}e^{-2\pi i\langle y,\xi\rangle}
      \,d(\pi_{\theta\#}\nu)(y)\\
    &=\int_{\mathbb R^n}e^{-2\pi i\langle \pi_\theta x,\xi\rangle}
      \,d\nu(x)\\
    &=\int_{\mathbb R^n}e^{-2\pi i\langle x,\xi\rangle}
      \,d\nu(x)
     =\widehat\nu(\xi).
\end{split}
\]
\end{proof}

Define $g_{\theta,0}$ on $V_\theta$ by
\[
    \widehat{g_{\theta,0}}(\xi)
    :=
    \left(1-\sum_{j\geq1}\Phi_{\theta,j}(\xi)\right)
    \widehat\mu(\xi),
    \qquad \xi\in V_\theta.
\]
By Lemma~\ref{lem:fourier-slice}, this is the low-frequency part of
$\pi_{\theta\#}\mu$.

For each $\tau$, the partition of unity \eqref{eq:eta-partition} gives
\[
    \sum_{T\in\mathbb T_\tau}M_{T,\theta}\mu
    =
    \mu*\check\psi_{\tau,\theta}.
\]
Combining this identity with \eqref{eq:net-indexed-partition} and
Lemma~\ref{lem:fourier-slice}, we obtain
\begin{equation}\label{eq:exact-wave-packet-decomp}
\pi_{\theta\#}\mu
=g_{\theta,0}
+\sum_{j\geq1}\sum_{\bbk}
\sum_{\tau\in\mathcal T_{j,\bbk}(\theta)}
\sum_{T\in\mathbb T_\tau}\pi_{\theta\#}M_{T,\theta}\mu
\end{equation}
in the sense of distributions on $V_\theta$.

\begin{lemma}[Wave packet localization]
\label{lem:wave-packet-localization}
Uniformly in $\theta$, for every
$\tau\in\mathcal T_{j,\bbk}(\theta)$,
$T\in\mathbb T_\tau$,
\begin{equation}\label{eq:L1-localization}
    \|M_{T,\theta}\mu\|_1
    \lesssim
    2^{C\delta j}\mu(100T)
    +\operatorname{RapDec}(2^j)\mu(\mathbb R^n).
\end{equation}
Moreover,
\[
    \sum_{\substack{T\in\mathbb T_\tau\\
                     100T\cap B^n(0,1)=\varnothing}}
    \|M_{T,\theta}\mu\|_1
    \leq
    \operatorname{RapDec}(2^j)\mu(\mathbb R^n).
\]
\end{lemma}

\begin{proof}
The first assertion is the standard wave packet localization estimate;
see \cite[Lemma~4]{gan2022restricted}.  The second follows by summing the rapidly
decaying tails of the convolution kernel \(\check\psi_{\tau,\theta}\)
over the finitely overlapping family \(T\in\mathbb T_\tau\).  The
estimates are uniform in \(\theta\) because the output boxes at nearby
parameters are uniformly comparable.
\end{proof}

Consequently, we omit below all packets satisfying
$100T\cap B^n(0,1)=\varnothing$.

\subsection{The good--bad decomposition}

We now classify the wave packets according to the $\mu$-mass of the
associated enlarged spatial planks $100T$.

Fix a small parameter $\delta>0$.  For an output $\bbk$, define
\begin{equation}\label{eq:A-bbk}
    A_{\bbk}
    :=
    \sum_{i=1}^J
    \frac{(m-m_i)(m+1-m_i)}{2}\,k_i.
\end{equation}
By the definition \eqref{eq:a-r} of the side-length exponents,
\begin{equation}\label{eq:tau-volume}
\begin{split}
    \sum_{r=1}^m a_r
    &=mj-
      \sum_{i=1}^J\sum_{r=m_i+1}^m(r-m_i)k_i\\
    &=mj-A_{\bbk},
    \qquad
    |\tau|\sim2^{mj-A_{\bbk}}
\end{split}
\end{equation}
for every output box $\tau$ associated with $\bbk$.  Let $\tau^*$ be
the dual plank.  Then
\begin{equation}\label{eq:T-volume}
    |\tau^*|\sim2^{-mj+A_{\bbk}}.
\end{equation}

Let $C_0$ be sufficiently large in terms of $n,m$ and the fixed curve. Recall \eqref{eq:alpha-zero}.  Define
\begin{equation}\label{eq:beta}
    \beta_{j,\bbk}
    :=2^{-\alpha_0j+A_{\bbk}+C_0\epsilon j}.
\end{equation}
A plank $T\in\mathbb T_\tau$ is called \emph{bad} if
\begin{equation}\label{eq:bad-definition}
    \mu(100T)\geq\beta_{j,\bbk},
\end{equation}
and \emph{good} otherwise.  Denote the corresponding collections by $\mathbb T_{\tau,b}$ and $\mathbb T_{\tau,g}$.

Using the exact $\theta$-adapted decomposition
\eqref{eq:exact-wave-packet-decomp}, define
\begin{equation}\label{eq:bad-good-distributions}
\begin{split}
    g_{\theta,b}
    &:=\sum_{j\geq1}\sum_{\bbk}
    \sum_{\tau\in\mathcal T_{j,\bbk}(\theta)}
    \sum_{T\in\mathbb T_{\tau,b}}
    \pi_{\theta\#}M_{T,\theta}\mu,\\
    g_{\theta,g}
    &:=\sum_{j\geq1}\sum_{\bbk}
    \sum_{\tau\in\mathcal T_{j,\bbk}(\theta)}
    \sum_{T\in\mathbb T_{\tau,g}}
    \pi_{\theta\#}M_{T,\theta}\mu.
\end{split}
\end{equation}
We refer to $g_{\theta,b}$ and $g_{\theta,g}$ as the bad part and the
good part of $\pi_{\theta\#}\mu$, respectively.  Thus, in the sense of
distributions on $V_\theta$,
\begin{equation}\label{eq:measure-decomposition}
    \pi_{\theta\#}\mu=g_{\theta,0}+g_{\theta,b}+g_{\theta,g}.
\end{equation}
Theorem \ref{thm:measure} follows from the following propositions.

\begin{proposition}[Bad part]\label{prop:bad1}
If $0<\epsilon<1$ and $\delta\ll\epsilon$, then there exists $1<q<2$,
independent of $\mu$, such that
\begin{equation}\label{eq:bad-goal}
    \int_0^1
    \|g_{\theta,b}\|_{L^q(V_\theta)}^q\,d\theta
    \lesssim_{\alpha,\epsilon} \mu(\mathbb R^n).
\end{equation}
\end{proposition}

\begin{proposition}[Good part]\label{prop:good1}
If
\begin{equation}
    \epsilon\ll\alpha-m
\end{equation}
and $\delta>0$ is sufficiently small in terms of $\epsilon$, then
\begin{equation}\label{eq:good-goal}
    \int_0^1
    \|g_{\theta,g}\|_{L^2(V_\theta)}^2\,d\theta
    \lesssim_{\alpha,\epsilon} \mu(\mathbb R^n).
\end{equation}
\end{proposition}

\begin{proof}[Completion of the proof of Theorem~\ref{thm:measure}]
Choose $\epsilon>0$ sufficiently small in terms of $\alpha-m$, and
then choose $\delta>0$ sufficiently small in terms of $\epsilon$. Let
$q\in(1,2)$ be given by Proposition~\ref{prop:bad1}. Propositions
\ref{prop:bad1} and \ref{prop:good1} imply that, for almost every
$\theta$,
\[
    g_{\theta,b}\in L^q(V_\theta),
    \qquad
    g_{\theta,g}\in L^2(V_\theta).
\]
The Fourier support of $g_{\theta,0}$ lies in a fixed ball. Since
$|\widehat\mu|\leq\mu(\mathbb R^n)$, Plancherel's theorem and
\eqref{eq:normalization} give
\[
    \int_0^1
    \|g_{\theta,0}\|_{L^2(V_\theta)}^2\,d\theta
    \lesssim
    \mu(\mathbb R^n)^2
    \leq
    \mu(\mathbb R^n).
\]
Consequently, the right-hand side of \eqref{eq:measure-decomposition}
is locally integrable for almost every $\theta$. Since
\eqref{eq:measure-decomposition} holds in the sense of distributions,
$\pi_{\theta\#}\mu$ is absolutely continuous. 

Since $1<q<2$, \eqref{eq:measure-decomposition} gives pointwise
\[
    (\pi_{\theta\#}\mu)^q
    \lesssim
    \pi_{\theta\#}\mu
    +|g_{\theta,0}+g_{\theta,g}|^2
    +|g_{\theta,b}|^q.
\]
Indeed, on the set where $\pi_{\theta\#}\mu\leq1$ the first term on the
right suffices; on its complement, either
$|g_{\theta,0}+g_{\theta,g}|\geq(\pi_{\theta\#}\mu)/2$ or
$|g_{\theta,b}|\geq(\pi_{\theta\#}\mu)/2$.
Therefore,
\[
\begin{split}
    \int_0^1
    \|\pi_{\theta\#}\mu\|_{L^q(V_\theta)}^q\,d\theta
    \lesssim{}&
    \mu(\mathbb R^n)
    +\int_0^1\|g_{\theta,0}\|_{L^2(V_\theta)}^2\,d\theta\\
    &+\int_0^1\|g_{\theta,g}\|_{L^2(V_\theta)}^2\,d\theta
    +\int_0^1\|g_{\theta,b}\|_{L^q(V_\theta)}^q\,d\theta\\
    \lesssim_\alpha{}&
    \mu(\mathbb R^n).
\end{split}
\]
This proves \eqref{eq:averaged-Lq} and completes the proof.
\end{proof}

Together with the reduction in the introduction, this proves
Theorem~\ref{thm:main}.

\section{The bad part}

In this section we prove Proposition~\ref{prop:bad1}.  The argument
combines a geometric covering estimate for projected wave packets with
the incidence estimate.  The badness condition controls the number of
packets associated with each output box, and the covering estimate
controls the number of $2^{-j}$-balls contributed by each packet.
After discretizing the parameter $\theta$, we apply the incidence
estimate to obtain decay in the annular scale $j$.  The wave packet
localization estimate gives an averaged $L^1$ estimate at each scale.
Interpolation with a trivial $L^\infty$ estimate then gives the required
$L^q$ bound. Recall Proposition \ref{prop:bad1}.

\begin{proposition}[Bad part]\label{prop:bad}
If $0<\epsilon<1$ and $\delta\ll\epsilon$, then there exists $1<q<2$,
independent of $\mu$, such that
\begin{equation}
    \int_0^1
    \|g_{\theta,b}\|_{L^q(V_\theta)}^q\,d\theta
    \lesssim_{\alpha,\epsilon} \mu(\mathbb R^n).
\end{equation}
\end{proposition}

We first estimate how many \(2^{-j}\)-balls are needed to cover the
projection of a wave packet. Recall \eqref{eq:A-bbk}.

\begin{lemma}[Projected wave packets]\label{lem:wave-packet}
For $\tau\in\mathcal T_{j,\bbk}(\theta)$ and $T\in\mathbb T_\tau$, the
set $\pi_\theta(100T)$ can be covered by
\begin{equation}\label{eq:cover-number}
    \lesssim2^{A_{\bbk}+C\delta j}
\end{equation}
balls of radius $2^{-j}$ in $V_\theta$.
\end{lemma}

\begin{proof}
The output box at $\theta_{\bbk}$ and the corresponding box at
$\theta$ are contained in fixed dilations of one another.  Their dual planks
are therefore comparable.  Projecting onto $V_\theta$ gives side lengths
$2^{-a_1},\ldots,2^{-a_m}$, up to the factor $2^{C\delta j}$.  Hence the
number of $2^{-j}$-balls required is
\[
    \lesssim2^{C\delta j}\prod_{r=1}^m2^{j-a_r}
    =2^{A_{\bbk}+C\delta j}.
\]
The equality follows by \eqref{eq:tau-volume}.
\end{proof}

\begin{proof}[Proof of Proposition~\ref{prop:bad}]
Fix $j\geq1$ and an output $\bbk$, and write
\begin{equation}\label{eq:bad-j-k-piece}
    \mathcal B_{j,\bbk}(\theta)
    :=
    \sum_{\tau\in\mathcal T_{j,\bbk}(\theta)}
    \sum_{T\in\mathbb T_{\tau,b}}
    \pi_{\theta\#}M_{T,\theta}\mu.
\end{equation}
Thus the bad planks occurring in $\mathcal B_{j,\bbk}(\theta)$ are
those attached to the output boxes based at the nearest
$\ell_{\bbk}$-net point $\theta_{\bbk}$. Note that
\begin{equation}
    g_{\theta,b}=
    \sum_{j\geq1}\sum_{\bbk}
   \mathcal B_{j,\bbk}(\theta).
\end{equation}
By the triangle inequality, we have
\begin{equation}\label{08014.1}
     \int_0^1\|g_{\theta,b}\|_{L^1(V_\theta)}\,d\theta
    \leq
    \sum_{j\geq1}\sum_{\bbk}
    \int_0^1
    \|\mathcal B_{j,\bbk}(\theta)\|_{L^1(V_\theta)}\,d\theta.
\end{equation}
Our primary goal is to prove that, for some $\epsilon_0 > 0$ independent of $j$ and $\bbk$,
\begin{equation}\label{0814.}
\begin{split}
    \int_0^1
    \|\mathcal B_{j,\bbk}(\theta)\|_{L^1(V_\theta)}\,d\theta
    \lesssim
    \mu(\mathbb R^n)2^{-(\epsilon_0-C\delta)j}.
\end{split}
\end{equation}
We will later interpolate this estimate with a trivial $L^\infty$ bound.

To prove \eqref{0814.}, we first bound the number of projected bad planks.  For a fixed output
box $\tau$, the family $\{100T:T\in\mathbb T_\tau\}$ has overlap at most
$2^{C\delta j}$.  Since every bad plank satisfies
$\mu(100T)\geq\beta_{j,\bbk}$, it follows that
\begin{equation}\label{eq:number-bad}
\begin{split}
    \#\mathbb T_{\tau,b}
    &\lesssim
    2^{C\delta j}
    \frac{\mu(\mathbb R^n)}{\beta_{j,\bbk}}\\
    &\lesssim
    \mu(\mathbb R^n)
    2^{\alpha_0j-A_{\bbk}-(C_0\epsilon-C\delta)j}.
\end{split}
\end{equation}
By Lemma~\ref{lem:wave-packet}, each $\pi_\theta(100T)$ can be covered by
at most
\[
    \lesssim 2^{A_{\bbk}+C\delta j}
\]
balls of radius $2^{-j}$ in $V_\theta$. Together with the preceding bound on the number of bad planks, this shows that their projections can be covered by at
most
\begin{equation}\label{eq:number-bad-balls}
    \lesssim
    \mu(\mathbb R^n)
    2^{\alpha_0j-(C_0\epsilon-C\delta)j}
    \leq
    \mu(\mathbb R^n)2^{\alpha_0j}
\end{equation}
balls of radius $2^{-j}$. \medskip

We now discretize the parameter $\theta$.  The integral of a
nonnegative function on $[0,1]$ is the average of the corresponding
 sums over all translates of the $2^{-j}$-lattice.  Hence by pigeonholing we may
choose a translated $2^{-j}$-net
$\Lambda_{2^{-j}}\subset[0,1]$ such that
\begin{equation}\label{eq:bad-discretization}
    \int_0^1
    \|\mathcal B_{j,\bbk}(\theta)\|_{L^1(V_\theta)}\,d\theta
    \lesssim
    2^{-j}
    \sum_{\vartheta\in\Lambda_{2^{-j}}}
    \|\mathcal B_{j,\bbk}(\vartheta)\|_{L^1(V_\vartheta)}.
\end{equation}
Fix $\vartheta\in\Lambda_{2^{-j}}$.  The family
$\mathcal T_{j,\bbk}(\vartheta)$ consists of the output boxes based at
the nearby $\ell_{\bbk}$-net point $\vartheta_{\bbk}$.
Although these boxes are indexed by $\vartheta_{\bbk}$ rather
than by $\vartheta$ itself, this causes no difficulty.  Indeed,
\[
    |\vartheta-\vartheta_{\bbk}|\lesssim\ell_{\bbk},
\]
so the output boxes at $\vartheta$ and at $\vartheta_{\bbk}$ are
essentially the same.  Passing to dual planks, every
$T\in\mathbb T_\tau$ is contained in a fixed dilation of the corresponding
$\vartheta$-adapted plank, and conversely.

Let
\[
    \mathbb T_{\vartheta,b}
    :=
    \bigcup_{\tau\in\mathcal T_{j,\bbk}(\vartheta)}
    \mathbb T_{\tau,b}.
\]
We define $U_\vartheta\subset V_\vartheta$ to be a union of
$2^{-j}$-balls covering the projections of all these bad planks:
\[
    \bigcup_{T\in\mathbb T_{\vartheta,b}}
    \pi_\vartheta(100T)
    \subset U_\vartheta.
\]
By \eqref{eq:number-bad-balls}, the set $U_\vartheta$ can be built
from
\[
    \lesssim \mu(\mathbb R^n)2^{\alpha_0j}
\]
balls of radius $2^{-j}$. By a standard finite decomposition, it is enough to consider a
pairwise disjoint collection of at most
\(\lesssim\mu(\mathbb R^n)2^{\alpha_0j}\) balls of radius \(2^{-j}\). This places us in the setting of the following incidence estimate.

\begin{proposition}[Incidence estimate]\label{prop:incidence}
Let $0<\alpha<n$ and $0<\alpha^*<\min\{m,\alpha\}$.  Suppose that $\nu$
is supported on $B^n(0,1)$ and $c_\alpha(\nu)\leq1$.  Let $\rho>0$ and
let $\Lambda_\rho$ be a $\rho$-net in $[0,1]$.  For each
$\theta\in\Lambda_\rho$, let $\mathbb D_\theta$ be a disjoint collection
of at most
\[
    \nu(\mathbb R^n)\rho^{-\alpha^*}
\]
balls of radius $\rho$ in $V_\theta$.  Then there exists $\epsilon_0>0$
such that
\begin{equation}\label{eq:incidence}
    \rho\sum_{\theta\in\Lambda_\rho}
    \pi_{\theta\#}\nu\left(\bigcup_{D\in\mathbb D_\theta}D\right)
    \lesssim_{\alpha,\alpha^*}
    \nu(\mathbb R^n)\rho^{\epsilon_0}.
\end{equation}
\end{proposition}

This is the consequence of Theorem~2.1 of \cite{gan2024restricted} recorded as
\cite[(2.7)]{gan2024restricted}.  Apply Proposition~\ref{prop:incidence} with
$\nu=\mu$, $\rho=2^{-j}$, and $\alpha^*=\alpha_0$ to each of the disjoint
subcollections above.  Summing the resulting estimates gives an exponent
$\epsilon_0>0$ such that
\begin{equation}\label{eq:bad-incidence}
    2^{-j}
    \sum_{\vartheta\in\Lambda_{2^{-j}}}
    \pi_{\vartheta\#}\mu(U_\vartheta)
    \lesssim
    \mu(\mathbb R^n)2^{-\epsilon_0j}.
\end{equation}

It remains to compare the bad wave packets with $U_\vartheta$.  By the
triangle inequality and \eqref{eq:L1-localization},
\begin{align}
\|\mathcal B_{j,\bbk}(\vartheta)\|_{L^1(V_\vartheta)}
&\leq
\sum_{\tau\in\mathcal T_{j,\bbk}(\vartheta)}
\sum_{T\in\mathbb T_{\tau,b}}
\|\pi_{\vartheta\#}M_{T,\vartheta}\mu\|_{L^1(V_\vartheta)}
\notag\\
&\lesssim
2^{C\delta j}
\sum_{\tau\in\mathcal T_{j,\bbk}(\vartheta)}
\sum_{T\in\mathbb T_{\tau,b}}
\mu(100T)
+\operatorname{RapDec}(2^j)\mu(\mathbb R^n).
\label{eq:bad-L1-localization}
\end{align}
Every plank in the last sum is contained in
$\pi_\vartheta^{-1}(U_\vartheta)$.  Since the families $\{100T\}$ have
overlap at most $2^{C\delta j}$, 
\[
\sum_{\tau\in\mathcal T_{j,\bbk}(\vartheta)}
\sum_{T\in\mathbb T_{\tau,b}}
\mu(100T)
\lesssim
2^{C\delta j}
\pi_{\vartheta\#}\mu(U_\vartheta).
\]
Consequently,
\begin{equation}\label{eq:bad-discrete-pointwise}
    \|\mathcal B_{j,\bbk}(\vartheta)\|_{L^1(V_\vartheta)}
    \lesssim
    2^{C\delta j}\pi_{\vartheta\#}\mu(U_\vartheta)
    +\operatorname{RapDec}(2^j)\mu(\mathbb R^n).
\end{equation}
Combining \eqref{eq:bad-discretization}, \eqref{eq:bad-incidence}, and
\eqref{eq:bad-discrete-pointwise}, we obtain
\begin{equation}\label{eq:bad-j-k}
\begin{split}
    \int_0^1
    \|\mathcal B_{j,\bbk}(\theta)\|_{L^1(V_\theta)}\,d\theta
    &\lesssim
    2^{C\delta j}2^{-j}
    \sum_{\vartheta\in\Lambda_{2^{-j}}}
    \pi_{\vartheta\#}\mu(U_\vartheta)
    +\operatorname{RapDec}(2^j)\mu(\mathbb R^n)\\
    &\lesssim
    \mu(\mathbb R^n)2^{-(\epsilon_0-C\delta)j}.
\end{split}
\end{equation}
Choose $\delta>0$ sufficiently small that $\epsilon_0-C\delta>0$.
Bounded overlap and Fourier inversion also give
\[
    \sup_{\theta\in[0,1]}
    \|\mathcal B_{j,\bbk}(\theta)\|_{L^\infty(V_\theta)}
    \lesssim 2^{Cj}\mu(\mathbb R^n).
\]
Hence, for $q>1$,
\[
    \int_0^1
    \|\mathcal B_{j,\bbk}(\theta)\|_{L^q(V_\theta)}^q\,d\theta
    \lesssim
    \mu(\mathbb R^n)^q
    2^{-(\epsilon_0-C\delta-C(q-1))j}.
\]
Choose $q\in(1,2)$ sufficiently close to $1$. By
\eqref{eq:normalization}, the preceding estimate gives
\[
    \int_0^1
    \|\mathcal B_{j,\bbk}(\theta)\|_{L^q(V_\theta)}^q\,d\theta
    \lesssim
    \mu(\mathbb R^n)2^{-cj}
\]
for some $c>0$. Since there are $O(j^C)$ outputs at scale $j$, 
\[
\begin{split}
    \left(
        \int_0^1
        \|g_{\theta,b}\|_{L^q(V_\theta)}^q\,d\theta
    \right)^{1/q}
    &\lesssim
    \mu(\mathbb R^n)^{1/q}
    \sum_{j\geq1}j^C2^{-cj/q}\\
    &\lesssim
    \mu(\mathbb R^n)^{1/q}.
\end{split}
\]
This proves Proposition~\ref{prop:bad}.
\end{proof}

\section{The good part}
In this section we prove Proposition~\ref{prop:good1}.  
Let us recall Proposition~\ref{prop:good1}.

\begin{proposition}[Good part]\label{prop:good}
If
\begin{equation}\label{eq:parameter-gap}
    \epsilon\ll\alpha-m
\end{equation}
and $\delta>0$ is sufficiently small in terms of $\epsilon$, then
\begin{equation}
    \int_0^1\|g_{\theta,g}\|_{L^2(V_\theta)}^2\,d\theta
    \lesssim_{\alpha,\epsilon}\mu(\mathbb R^n).
\end{equation}
\end{proposition}

\begin{proof}[Proof of Proposition~\ref{prop:good}]
The proof is divided into three steps.
\medskip

\noindent\textbf{Step 1: Orthogonality and discretization in $\theta$.} Recall that we have defined $\mathcal T_{j,\bbk}(\theta)$ at the end of Section \ref{sec2}.
For $j\geq1$ and an output $\bbk$, 
\begin{equation}\label{eq:good-j-k-piece}
    \mathcal G_{j,\bbk}(\theta)
    :=
    \sum_{\tau\in\mathcal T_{j,\bbk}(\theta)}
    \sum_{T\in\mathbb T_{\tau,g}}
    \pi_{\theta\#}M_{T,\theta}\mu,
\end{equation}
and write
\begin{equation}\label{eq:S-j-k}
    \mathcal S_{j,\bbk}
    :=
    \int_0^1
    \|\mathcal G_{j,\bbk}(\theta)\|_{L^2(V_\theta)}^2\,d\theta.
\end{equation}
Note that
\begin{equation}
     g_{\theta,g}= \sum_j \sum_{\bbk}  \mathcal G_{j,\bbk}(\theta). 
\end{equation}

We first separate the sums over $j$ and $\bbk$.  For finite partial sums,
the Fourier slice theorem and the $m$-dimensional Plancherel theorem give
\begin{align}
&\int_0^1
\left\|
    \sum_{j,\bbk}\mathcal G_{j,\bbk}(\theta)
\right\|_{L^2(V_\theta)}^2\,d\theta
\notag\\
&\qquad=
\int_0^1\int_{V_\theta}
\left|
    \sum_{j,\bbk}
    \sum_{\tau\in\mathcal T_{j,\bbk}(\theta)}
    \sum_{T\in\mathbb T_{\tau,g}}
    \widehat{M_{T,\theta}\mu}(\xi)
\right|^2d\xi\,d\theta.
\label{eq:full-good-plancherel}
\end{align}
The frequency supports have bounded overlap in the annular index $j$.
For each fixed $j$, there are $O(j^{C})$ possible outputs $\bbk$;
therefore Cauchy--Schwarz in the $\bbk$-sum yields
\begin{equation}\label{eq:full-good-reduction}
\int_0^1
\left\|
    \sum_{j,\bbk}\mathcal G_{j,\bbk}(\theta)
\right\|_{L^2(V_\theta)}^2\,d\theta
\lesssim
\sum_{j\geq1}j^{C}\sum_{\bbk}\mathcal S_{j,\bbk}.
\end{equation}
Thus it remains to obtain a summable estimate for each
$\mathcal S_{j,\bbk}$.
\medskip

Fix $j\geq1$ and an output $\bbk$.  The $\ell_{\bbk}$-net
$\Theta_{\bbk}$ divides $[0,1]$ into output intervals $I$ of length
comparable to $\ell_{\bbk}$.  For $\theta\in I$, the nearest net point
$\theta_{\bbk}$ is fixed, and hence so is the collection of output boxes.
We denote this common collection by $\mathcal T_{j,\bbk}(I)$.  As $I$
ranges over the output intervals, let
\[
    \mathcal T^{\mathrm{net}}_{j,\bbk}
    :=
    \bigsqcup_I\mathcal T_{j,\bbk}(I)
\]
denote the resulting indexed family of output boxes. 

For a fixed \(\theta\), the Fourier slice theorem and the
\(m\)-dimensional Plancherel theorem give
\begin{align}
\|\mathcal G_{j,\bbk}(\theta)\|_{L^2(V_\theta)}^2
&=
\int_{V_\theta}
\left|
\sum_{\tau\in\mathcal T_{j,\bbk}(\theta)}
\sum_{T\in\mathbb T_{\tau,g}}
\widehat{M_{T,\theta}\mu}(\xi)
\right|^2\,d\xi.
\label{eq:good-plancherel-one}
\end{align}
For fixed \(\theta\), \(j\), and \(\bbk\), the collection
\(\mathcal T_{j,\bbk}(\theta)\) consists of \(O_n(1)\) rectangular sign
components. Hence, by Cauchy--Schwarz in the \(\tau\)-sum,
\begin{equation}\label{eq:tau-orthogonality}
\|\mathcal G_{j,\bbk}(\theta)\|_{L^2(V_\theta)}^2
\lesssim
\sum_{\tau\in\mathcal T_{j,\bbk}(\theta)}
\left\|
\sum_{T\in\mathbb T_{\tau,g}}
\pi_{\theta\#}M_{T,\theta}\mu
\right\|_{L^2(V_\theta)}^2.
\end{equation}

Fix one $\tau$.  The planks $T\in\mathbb T_\tau$ form a finitely
overlapping tiling.  Since $|\theta-\theta_{\bbk}|=O(\ell_{\bbk})$, the corresponding projected planks
$\pi_\theta(100T)$ have overlap at most $2^{C\delta j}$.  Applying
Cauchy--Schwarz, followed again by the $m$-dimensional
Plancherel theorem, yields
\begin{align}
\left\|
\sum_{T\in\mathbb T_{\tau,g}}
\pi_{\theta\#}M_{T,\theta}\mu
\right\|_{L^2(V_\theta)}^2
&\lesssim
2^{C\delta j}
\sum_{T\in\mathbb T_{\tau,g}}
\|\pi_{\theta\#}M_{T,\theta}\mu\|_{L^2(V_\theta)}^2
\notag\\
&=
2^{C\delta j}
\sum_{T\in\mathbb T_{\tau,g}}
\int_{V_\theta}|\widehat{M_{T,\theta}\mu}(\xi)|^2\,d\xi
\notag\\
&\lesssim
2^{C\delta j}
\sum_{T\in\mathbb T_{\tau,g}}
\|M_{T,\theta}\mu\|_2^2.
\label{eq:T-orthogonality}
\end{align}
In the last step we used the usual uncertainty-principle estimate in the
directions perpendicular to $V_\theta$; the enlargement of the packets
contributes only the displayed $2^{C\delta j}$ loss.

Combining \eqref{eq:tau-orthogonality} and
\eqref{eq:T-orthogonality}, and then integrating over the output
intervals, gives
\begin{equation}\label{eq:S-integrated-theta-dependent}
\mathcal S_{j,\bbk}
\lesssim
2^{C\delta j}
\sum_I\int_I
\sum_{\tau\in\mathcal T_{j,\bbk}(I)}
\sum_{T\in\mathbb T_{\tau,g}}
\|M_{T,\theta}\mu\|_2^2
\,d\theta.
\end{equation}
For each $I$, choose $\theta_I^*\in I$ so that
\[
\int_I
\sum_{\tau\in\mathcal T_{j,\bbk}(I)}
\sum_{T\in\mathbb T_{\tau,g}}
\|M_{T,\theta}\mu\|_2^2\,d\theta
\leq
2|I|
\sum_{\tau\in\mathcal T_{j,\bbk}(I)}
\sum_{T\in\mathbb T_{\tau,g}}
\|M_{T,\theta_I^*}\mu\|_2^2.
\]
For $\tau\in\mathcal T_{j,\bbk}(I)$, set
\[
    \psi_\tau^*:=\psi_{\tau,\theta_I^*},
    \qquad
    M_T\mu:=M_{T,\theta_I^*}\mu
    =\eta_T(\mu*\check\psi_\tau^*).
\]
Thus all multipliers and packets are fixed on each output interval.  Since
$|I|\lesssim\ell_{\bbk}$, we obtain
\begin{equation}\label{eq:S-after-freezing}
\mathcal S_{j,\bbk}
\lesssim
2^{C\delta j}\ell_{\bbk}
\sum_I
\sum_{\tau\in\mathcal T_{j,\bbk}(I)}
\sum_{T\in\mathbb T_{\tau,g}}
\|M_T\mu\|_2^2.
\end{equation}

\medskip
\noindent\textbf{Step 2: Pigeonholing and two estimates.}
For every packet in \eqref{eq:S-after-freezing}, define
\begin{equation}\label{eq:f-T}
    f_T(y)
    :=
    \int_{\mathbb R^n}
    \eta_T(x)\overline{M_T\mu(x)}
    \check\psi_\tau^*(x-y)\,dx.
\end{equation}
By Fubini and the identity
\[
    M_T\mu=\eta_T(\mu*\check\psi_\tau^*),
\]
we have
\begin{equation}\label{eq:fubini}
    \|M_T\mu\|_2^2=\int f_T\,d\mu.
\end{equation}
Moreover, the Fourier transform in the $y$-variable of
$\check\psi_\tau^*(x-y)$ is supported in $-2\tau$. Since
$\psi_\tau^*$ is uniformly adapted to $\tau$, the standard kernel
bounds give
\[
    \|\check\psi_\tau^*\|_1\lesssim1,
    \qquad
    \|\check\psi_\tau^*\|_\infty
    \lesssim|\tau|.
\]
Together with \eqref{eq:L1-localization}, these imply
\begin{equation}
    \operatorname{supp}\widehat f_T
    \subset 2(-\tau), 
\end{equation}
\begin{equation}\label{0818516}
    \|f_T\|_2
    \lesssim \|M_T\mu\|_2,
\end{equation}
and
\begin{equation}\label{eq:f-T-properties}
\begin{aligned}
    \|f_T\|_\infty
    &\leq
    \|\check\psi_\tau^*\|_\infty
    \|\eta_T M_T\mu\|_1\\
    &\lesssim
    |\tau|\,\|M_T\mu\|_1\\
    &\lesssim
    |\tau|
    \left(
        2^{C\delta j}\mu(100T)
        +\operatorname{RapDec}(2^j)\mu(\mathbb R^n)
    \right)\\
    &\lesssim
    2^{C\delta j}|\tau|\mu(100T)
    +\operatorname{RapDec}(2^j)\mu(\mathbb R^n).
\end{aligned}
\end{equation}
The same kernel bounds show that $f_T$ decreases rapidly away from
$100T$. \medskip

For the rest of the section, we set
\begin{equation}\label{eq:p}
    p=n(n+1).
\end{equation}
We use the two standard dyadic pigeonholings.
First group the packets according to the size of $\|f_T\|_p$.  Next tile
$B^n(0,1)$ by balls of radius $2^{-j}$ and group those balls according to
the number of selected planks meeting them.   We thereby obtain a subcollection
\[
    \mathbb W
    =\bigcup_{\tau\in\mathcal T^{\mathrm{net}}_{j,\bbk}}
      \mathbb W_\tau,
    \qquad
    \mathbb W_\tau\subset\mathbb T_{\tau,g},
\]
a dyadic number $M$, and a union $Y$ of disjoint $2^{-j}$-balls such that
\[
    \|f_T\|_p\sim\|f_{T'}\|_p
    \qquad (T,T'\in\mathbb W),
\]
every ball in $Y$ meets $\sim M$ planks from $\mathbb W$, and
\begin{equation}\label{eq:energy-captured-direct}
\sum_I
\sum_{\tau\in\mathcal T_{j,\bbk}(I)}
\sum_{T\in\mathbb T_{\tau,g}}
\|M_T\mu\|_2^2
\lesssim
2^{C\epsilon j}
\int_Y\left|\sum_{T\in\mathbb W}f_T\right|d\mu.
\end{equation}
Combining \eqref{eq:S-after-freezing} and
\eqref{eq:energy-captured-direct}, we obtain
\begin{equation}\label{eq:S-to-selected-integral}
    \mathcal S_{j,\bbk}
    \lesssim
    2^{C(\epsilon+\delta)j}\ell_{\bbk}
    \int_Y\left|\sum_{T\in\mathbb W}f_T\right|d\mu.
\end{equation}

We will use the following two estimates for the integral in
\eqref{eq:S-to-selected-integral}:
\begin{equation}\label{eq:first-bound}
    \int_Y\left|\sum_{T\in\mathbb W}f_T\right|d\mu
    \lesssim
    2^{C\delta j}M|\tau|\beta_{j,\bbk}\mu(\mathbb R^n),
\end{equation}
and
\begin{equation}\label{eq:second-bound}
\begin{split}
\int_Y\left|\sum_{T\in\mathbb W}f_T\right|d\mu
\lesssim{}&
2^{C\epsilon j}
\left(\frac{2^{K_{\bbk}}}{M}\right)^{1-\frac2p}
2^{\sum_i\left(1-\frac{(n_i-m_i)(n_i+1-m_i)}p\right)k_i}\\
&\times
|\tau|^{1-\frac2p}
\beta_{j,\bbk}^{1-\frac2p}
2^{2j(n-\alpha)/p}\mu(\mathbb R^n).
\end{split}
\end{equation}
We prove these estimates in order.

\smallskip
\noindent\emph{Proof of \eqref{eq:first-bound}.}
Every $T\in\mathbb W$ is good, so
\[
    \mu(100T)<\beta_{j,\bbk}.
\]
The last estimate in \eqref{eq:f-T-properties}, the rapid decay of $f_T$
away from $100T$, and the fact that every ball in $Y$ meets $\sim M$ selected
planks give
\[
    \left|\sum_{T\in\mathbb W}f_T(y)\right|
    \lesssim
    2^{C\delta j}M|\tau|\beta_{j,\bbk}
    +\operatorname{RapDec}(2^j)\mu(\mathbb R^n),
    \qquad y\in Y.
\]
Integrating against $\mu$ proves \eqref{eq:first-bound}.

\smallskip
\noindent\emph{Proof of \eqref{eq:second-bound}.}
Choose a Schwartz function $\rho$ whose Fourier transform equals one on a
sufficiently large fixed ball, and put
$\rho_j(x)=2^{jn}\rho(2^jx)$. By local constancy, after replacing $Y$ by
a fixed enlargement, H\"older's inequality with exponents $p$ and
$p/(p-1)$ gives
\begin{equation}\label{eq:holder}
\begin{split}
\int_Y\left|\sum_{T\in\mathbb W}f_T\right|d\mu
\lesssim{}&
\left\|\sum_{T\in\mathbb W}f_T\right\|_{L^p(Y)}\\
&\times
\left(
    \int_Y
    \bigl(\mu*|\rho_j|\bigr)^{p/(p-1)}
\right)^{1-1/p}.
\end{split}
\end{equation}
The Frostman condition gives
\begin{equation}\label{eq:frostman-convolution}
    \|\mu*\left|\rho_j \right|\|_\infty\lesssim2^{j(n-\alpha)}.
\end{equation}
Furthermore, every $2^{-j}$-ball in $Y$ meets approximately $M$ planks,
so
\[
    \mathbf1_Y\lesssim \frac1M \sum_{T\in\mathbb W}\mathbf1_{100T}.
\]
Since the selected planks are good,
\begin{equation}\label{eq:Y-mass}
    \int_Y\mu*\left|\rho_j\right|
    \lesssim
    \frac1M\sum_{T\in\mathbb W}\mu(100T)
    \lesssim
    \frac{|\mathbb W|}{M}\beta_{j,\bbk}.
\end{equation}

Since $\|\rho_j\|_1\lesssim1$, we also have
\[
    \int_Y\mu*|\rho_j|
    \lesssim
    \mu(\mathbb R^n).
\]
Hence, by interpolation between $L^1(Y)$ and $L^\infty(Y)$,
\begin{align}
&\left(
    \int_Y
    \bigl(\mu*|\rho_j|\bigr)^{p/(p-1)}
\right)^{1-1/p}
\notag\\
&\qquad\lesssim
\|\mu* \left|\rho_j \right|\|_\infty^{1/p}
\left(\int_Y\mu* \left|\rho_j\right|\right)^{1-1/p}
\notag\\
&\qquad\lesssim
2^{j(n-\alpha)/p}
\min\left\{
    \mu(\mathbb R^n),
    \frac{|\mathbb W|}{M}\beta_{j,\bbk}
\right\}^{1-\frac1p}
\notag\\
&\qquad\lesssim
2^{j(n-\alpha)/p}
\mu(\mathbb R^n)^{1/2}
\left(
    \frac{|\mathbb W|}{M}\beta_{j,\bbk}
\right)^{\frac12-\frac1p}.
\label{eq:measure-factor}
\end{align}

Each output interval corresponds to one point of the $\ell_{\bbk}$-net,
and each net point contributes only $O_n(1)$ output boxes.  Consequently,
\begin{equation}\label{eq:number-net-output-boxes}
    \#\mathcal T^{\mathrm{net}}_{j,\bbk}
    \lesssim \ell_{\bbk}^{-1}
    =2^{K_{\bbk}}.
\end{equation}

\begin{proposition}[Decoupling associated with an output]
\label{prop:output-decoupling}
Let $p=n(n+1)$ and $j\geq1$, and let $\bbk$ be an output.  Let
$\mathcal T^{\mathrm{net}}_{j,\bbk}$ be the net-indexed family above.
Suppose that $F_\tau\in L^p(\mathbb R^n)$ and
\[
    \operatorname{supp}\widehat{F_\tau}
    \subset C\tau,
    \qquad
    \tau\in\mathcal T^{\mathrm{net}}_{j,\bbk},
\]
where $C$ is fixed.  Then, for every $\eta>0$,
\begin{equation}\label{eq:output-decoupling}
\left\|
    \sum_{\tau\in\mathcal T^{\mathrm{net}}_{j,\bbk}}F_\tau
\right\|_{L^p(\mathbb R^n)}
\lesssim_\eta
2^{\eta j}
2^{\sum_{i=1}^J
\left(\frac12-\frac{(n_i-m_i)(n_i+1-m_i)}{2p}\right)k_i}
\left(
    \sum_{\tau\in\mathcal T^{\mathrm{net}}_{j,\bbk}}
    \|F_\tau\|_p^2
\right)^{1/2}.
\end{equation}
The same estimate holds for the reflected family
$-\mathcal T^{\mathrm{net}}_{j,\bbk}$.
\end{proposition}

Proposition~\ref{prop:output-decoupling} is obtained by iterating the
$\ell^2$ form of the stagewise decoupling argument underlying
\cite[Claims~3.1 and~4.1]{gan2024restricted}, with the stage scales written as
$2^{-k_i}$; see Appendix~\ref{app:frequency-decomposition}. \medskip

Apply Proposition~\ref{prop:output-decoupling}, with its loss chosen to be
$\epsilon$, to
\[
    F_\tau:=\sum_{T\in\mathbb W_\tau}f_T.
\]
We use the reflected version because
$\operatorname{supp}\widehat f_T\subset2(-\tau)$.  
\begin{align}
\left\|\sum_{T\in\mathbb W}f_T\right\|_p
&\lesssim
2^{C\epsilon j}
2^{\sum_i\left(\frac12-\frac{(n_i-m_i)(n_i+1-m_i)}{2p}\right)k_i}
\left(\sum_{\tau } \Big\| \sum_{T \in \mathbb W_{\tau}} f_T \Big\|_p^2\right)^{1/2}.
\end{align}
By
\eqref{eq:number-net-output-boxes}, the bounded overlap of the packets with
a fixed $\tau$, and the dyadic constancy of $\|f_T\|_p$, we obtain
\begin{align}
\left\|\sum_{T\in\mathbb W}f_T\right\|_p
&\lesssim
2^{C\epsilon j}
2^{\sum_i\left(\frac12-\frac{(n_i-m_i)(n_i+1-m_i)}{2p}\right)k_i}
\left(\frac{2^{K_{\bbk}}}{|\mathbb W|}\right)^{\frac12-\frac1p}
\left(\sum_{T\in\mathbb W}\|f_T\|_p^2\right)^{1/2}.
\label{eq:decoupling-W}
\end{align}
By Bernstein's inequality, \eqref{0818516}, and
\eqref{eq:energy-captured-direct},
\begin{align}
\left(\sum_{T\in\mathbb W}\|f_T\|_p^2\right)^{1/2}
&\lesssim
|\tau|^{\frac12-\frac1p}
\left(\sum_{T\in\mathbb W}\|f_T\|_2^2\right)^{1/2}
\notag\\
&\lesssim
2^{C\epsilon j}|\tau|^{\frac12-\frac1p}
\left(
\int_Y\left|\sum_{T\in\mathbb W}f_T\right|d\mu
\right)^{1/2}.
\label{eq:bernstein-and-capture}
\end{align}
Substituting \eqref{eq:measure-factor}, \eqref{eq:decoupling-W}, and
\eqref{eq:bernstein-and-capture} into \eqref{eq:holder}, and then
squaring, proves \eqref{eq:second-bound}.

\medskip
\noindent\textbf{Step 3: Completion of the proof.}
Raise \eqref{eq:first-bound} to the power $1-2/p$ and multiply it by
\eqref{eq:second-bound}.  The powers of $M$ cancel.  Combining the result
with \eqref{eq:S-to-selected-integral}, we obtain
\begin{equation}\label{eq:S-master}
\begin{split}
\mathcal S_{j,\bbk}^{2-\frac2p}
\lesssim{}&
2^{C\epsilon j}
\ell_{\bbk}^{2-\frac2p}
2^{(1-\frac2p)K_{\bbk}}
2^{\sum_i\left(1-\frac{(n_i-m_i)(n_i+1-m_i)}p\right)k_i}
2^{2j(n-\alpha)/p}\\
&\times
|\tau|^{2-\frac4p}
\beta_{j,\bbk}^{2-\frac4p}
\mu(\mathbb R^n)^{2-\frac2p}.
\end{split}
\end{equation}
Moreover,
\[
2^{\sum_i\left(1-\frac{(n_i-m_i)(n_i+1-m_i)}p\right)k_i}
=
2^{K_{\bbk}-\frac1p
\sum_i(n_i-m_i)(n_i+1-m_i)k_i}.
\]
Together with $\ell_{\bbk}=2^{-K_{\bbk}}$, this shows explicitly that the
powers of $K_{\bbk}$ cancel:
\[
    -\left(2-\frac2p\right)K_{\bbk}
    +\left(1-\frac2p\right)K_{\bbk}
    +K_{\bbk}=0.
\]
Using \eqref{eq:tau-volume} and \eqref{eq:beta}, we have
\[
    |\tau|\beta_{j,\bbk}
    \sim 2^{(m-\alpha_0)j+C_0\epsilon j}.
\]
Thus the powers of $A_{\bbk}$ also cancel, and hence
\begin{equation}\label{eq:S-exponent-one}
\mathcal S_{j,\bbk}^{2-\frac2p}
\lesssim
2^{C\epsilon j}
2^{\frac{2(n-\alpha)j-
\sum_i(n_i-m_i)(n_i+1-m_i)k_i}{p}}
2^{\left(2-\frac4p\right)(m-\alpha_0)j}
\mu(\mathbb R^n)^{2-\frac2p}.
\end{equation}

It remains to rewrite the sum in the middle exponent.  Summing
\eqref{eq:output-identities} over $q=m+1,\ldots,n$ and reversing the order
of summation, we obtain
\begin{align*}
    (n-m)j
    &=
    \sum_{i=1}^J
    \left(\sum_{q=m+1}^{n_i}(q-m_i)\right)k_i + O_n(1)\\
    &=
    \frac12
    \sum_{i=1}^J(n_i-m_i)(n_i+1-m_i)k_i
    -A_{\bbk} + O_n(1).
\end{align*}
Therefore
\begin{equation}\label{eq:output-scale-identity}
    \sum_{i=1}^J(n_i-m_i)(n_i+1-m_i)k_i
    =2(n-m)j+2A_{\bbk} + O_n(1).
\end{equation}
Substituting \eqref{eq:output-scale-identity} into
\eqref{eq:S-exponent-one} gives
\begin{equation}\label{eq:S-exponent-two}
\mathcal S_{j,\bbk}^{2-\frac2p}
\lesssim
2^{-2A_{\bbk}/p}
2^{-\frac2p\left[(\alpha-m)-(p-2)(m-\alpha_0)\right]j}
2^{C\epsilon j}
\mu(\mathbb R^n)^{2-\frac2p}.
\end{equation}
Substituting $m-\alpha_0=\epsilon$, and then choosing
$\epsilon\ll\alpha-m$ and $\delta\ll\epsilon$ sufficiently small, we
obtain some $c>0$ such that
\begin{equation}\label{eq:S-decay}
    \mathcal S_{j,\bbk}
    \lesssim
    \mu(\mathbb R^n)2^{-cj}.
\end{equation}
Returning to \eqref{eq:full-good-reduction}, and using again that the
number of outputs at scale $j$ is polynomial in $j$, we obtain
\[
\begin{aligned}
\int_0^1\|g_{\theta,g}\|_{L^2(V_\theta)}^2\,d\theta
&\lesssim
\sum_{j\geq1}j^{C}\sum_{\bbk}\mathcal S_{j,\bbk}\\
&\lesssim
\mu(\mathbb R^n)
\sum_{j\geq1}j^{C}2^{-cj}\\
&\lesssim
\mu(\mathbb R^n).
\end{aligned}
\]
This proves Proposition~\ref{prop:good}.
\end{proof}

\appendix
\renewcommand{\theequation}{\Alph{section}.\arabic{equation}}

\section{Multiscale frequency decomposition and iterated decoupling}
\label{app:frequency-decomposition}

This appendix explains how the decomposition in
\cite[Sections~3--5]{gan2024restricted} gives both the  frequency
decomposition in Proposition~\ref{prop:Vtheta-decomposition} and the
 decoupling estimate in
Proposition~\ref{prop:output-decoupling}.  We first describe the region to be
decomposed at each fixed parameter value, then carry out the first step and
the inductive step, including the corresponding decoupling estimates and
rescalings, and finally identify their outputs with the objects used in
Section~2.

At a fixed $t$, we use the derivative coordinates
\[
    \gamma^{(1)}(t),\ldots,\gamma^{(n)}(t).
\]
By Lemma~\ref{lem:frame-comparison}, up to $O_n(1)$ subdivisions into sign
components and fixed enlargements, it is enough to carry out the
decomposition in these coordinates.
Fix a dyadic frequency scale $2^j$ and put $\rho:=2^{-j}$.  After rescaling
frequency space by a factor of $2^{-j}$, the region to be decomposed has the form
\begin{equation}\label{eq:app-initial-region-new}
    \left\{
        \sum_{r=1}^n\lambda_r\gamma^{(r)}(t):
        \begin{array}{l}
        |\lambda_r|\lesssim1\quad(1\leq r\leq m),\\
        |\lambda_q|\lesssim\rho\quad(m<q\leq n),\\
        \displaystyle\sum_{r=1}^m|\lambda_r|\sim1
        \end{array}
    \right\}.
\end{equation}
In particular, at every fixed $t$, the coefficient range in each order
$q>m$ in
\eqref{eq:app-initial-region-new} has size $\rho$.  All the decompositions
below are carried out separately for each $t$ by smooth partitions of unity.  We
may assume that $j$ is sufficiently large.  

\subsection{The first stage}

Set $n_1:=n$.  Following \cite[(3.1)--(3.3)]{gan2024restricted}, divide
\eqref{eq:app-initial-region-new} according to the largest index $m_1\leq m$
whose coefficient has size comparable to one.  On the piece indexed by
$m_1$,
\begin{equation}\label{eq:app-first-m-new}
    |\lambda_{m_1}|\sim1,
    \qquad
    |\lambda_u|\ll1
    \quad(m_1<u\leq m).
\end{equation}

Let $s_{1,\mathrm{end}}$ be the smallest dyadic number not smaller than
$\rho^{1/(n-m_1)}$; thus
\[
    \rho^{1/(n-m_1)}
    \leq s_{1,\mathrm{end}}
    <2\rho^{1/(n-m_1)}.
\]
For dyadic $s_{1,\mathrm{end}}\leq s_1\ll1$, divide the $m_1$-piece
according to
\begin{equation}\label{eq:app-first-s-new}
    \frac{|\lambda_{m_1+1}|}{s_1}
    +\frac{|\lambda_{m_1+2}|}{s_1^2}
    +\cdots+
    \frac{|\lambda_m|}{s_1^{m-m_1}}
    \sim1.
\end{equation}
At $s_1=s_{1,\mathrm{end}}$, replace $\sim1$ by $\lesssim1$;
this is the endpoint piece.  If $m_1=m$, the sum in
\eqref{eq:app-first-s-new} is empty.  In this case the $m_1$-piece is not
further decomposed in frequency according to $s_1$, and we set
$s_1=s_{1,\mathrm{end}}$. 
This is the partition in \cite[(3.4)--(3.8)]{gan2024restricted}.

For later comparison, at every fixed $t$ the piece selected by
\eqref{eq:app-first-m-new} and \eqref{eq:app-first-s-new} is contained in
\begin{equation}\label{eq:app-first-region-new}
\left\{
    \sum_{r=1}^n \lambda_r\gamma^{(r)}(t):
    \begin{array}{ll}
    |\lambda_r|\lesssim1,
        &1\leq r<m_1,\\
    |\lambda_{m_1}|\sim1,
        &\\
    |\lambda_r|\lesssim s_1^{r-m_1},
        &m_1<r\leq m,\\
    |\lambda_q|\lesssim\rho,
        &m<q\leq n
    \end{array}
\right\}.
\end{equation}

Fix $m_1$ and $s_1$.  Divide $[0,1]$ into intervals $I$ of length
comparable to $s_1$,
and choose a point $t_I\in I$.  After splitting according to the sign of $\lambda_{m_1}$, fix
$\sigma\in\{-1,1\}$ accordingly and consider the box
\begin{equation}\label{eq:app-first-box-new}
\left\{
    \sum_{r=1}^n \mu_r\gamma^{(r)}(t_I):
    \begin{array}{ll}
    |\mu_r|\lesssim1,
        &1\leq r<m_1,\\
    \sigma\mu_{m_1}\sim1,
        &\\
    |\mu_r|\lesssim s_1^{r-m_1},
        &m_1<r\leq n
    \end{array}
\right\}.
\end{equation}
Since
\[
    \rho\leq s_1^{n-m_1}\leq s_1^{q-m_1},
    \qquad m<q\leq n,
\]
the same Taylor-expansion argument used in Subsection~2.2 shows that, for
fixed $m_1$, $s_1$, and $\sigma$, the union over $t\in I$ of the
corresponding regions is contained in a fixed enlargement of
the box in \eqref{eq:app-first-box-new}.  

Put $p:=n(n+1)$.  Fix $m_1$, $s_1$, and
$\sigma\in\{-1,1\}$.  For each interval $I$ above, let
$F_I\in L^p(\mathbb R^n)$ have Fourier support in a fixed enlargement of
the box in \eqref{eq:app-first-box-new} based at $t_I$ and having sign
$\sigma$. The Fourier supports above form, up to a fixed linear change of
coordinates, a cylinder with $m_1$ flat directions and an
$(n_1-m_1)$-dimensional non-degenerate curve factor. Applying the
$\ell^2$ decoupling inequality to the curved factor and handling the
flat directions by Fubini gives, for every $\eta>0$,
\begin{equation}\label{eq:app-one-stage-new}
\left\|
    \sum_I F_I
\right\|_p
\lesssim_\eta
\rho^{-\eta}
s_1^{-\left(
    \frac12-
    \frac{(n_1-m_1)(n_1+1-m_1)}{2p}
\right)}
\left(
    \sum_I\|F_I\|_p^2
\right)^{1/2}.
\end{equation}
Although \cite[Claim~3.1]{gan2024restricted} is stated with an $\ell^p$ sum,
\eqref{eq:app-one-stage-new} follows by retaining the square sum supplied
by the standard $\ell^2$ curve decoupling \cite{MR3548534}, instead of applying
the final $\ell^2$-to-$\ell^p$ H\"older inequality.  

After applying \eqref{eq:app-one-stage-new}, we continue the argument
separately on each interval $I$ and rescale the regions selected by
\eqref{eq:app-first-m-new} and \eqref{eq:app-first-s-new} as $t$
ranges over $I$.  Writing
$I=[t_0,t_0+|I|]$, rescale $I$ affinely to $[0,1]$ by setting
\[
    u:=\frac{t-t_0}{|I|}.
\]
The accompanying change of frequency variables also replaces $\gamma$
by a rescaled curve $\widetilde\gamma$.  Under the frequency change of variables in
\cite[(3.25)]{gan2024restricted} and the coefficient change in
\cite[(3.30)]{gan2024restricted}, let $\lambda_r^{\mathrm{new}}$ denote
the coefficient variable passed to the next stage.  Up to fixed
constants, $\lambda_r^{\mathrm{new}}$ is $\lambda_r$ for $r\leq m_1$
and $s_1^{-(r-m_1)}\lambda_r$ for $r>m_1$.  We then relabel
$\lambda_r^{\mathrm{new}}$ as $\lambda_r$.  In these relabeled
variables,
the terms of orders $r\geq m_1$ take the form
\[
    \lambda_{m_1}\widetilde\gamma^{(m_1)}(u)
    +\cdots+
    \lambda_n\widetilde\gamma^{(n)}(u),
\]
as in \cite[(3.40)]{gan2024restricted}.  The directions of orders $r<m_1$ are flat and are always handled by
Fubini.  We therefore suppress their precise form after the change of
variables and, for notational convenience, continue to denote them by
$\widetilde\gamma^{(r)}(u)$.  For $r<m_1$, this notation refers to the
transformed flat direction rather than the actual $r$-th derivative of
the rescaled curve.  Only the terms with $r\geq m_1$ are written
literally in the derivative coordinates of $\widetilde\gamma$.  We use
the same convention at every later stage.  With this convention, we rename $\widetilde\gamma$ as $\gamma$.  Set
\[
    \ell_1:=s_1,
    \qquad
    b_q^{[1]}:=\rho s_1^{-(q-m_1)},
    \qquad m<q\leq n_1.
\]
In these renamed variables, the region corresponding to $u$ is
contained in
\[
\left\{
    \sum_{r=1}^{n_1}\lambda_r\gamma^{(r)}(u):
    \begin{array}{ll}
    |\lambda_r|\lesssim1,
        &1\leq r<m_1,\\
    |\lambda_{m_1}|\sim1,
        &\\
    |\lambda_r|\lesssim1,
        &m_1<r\leq m,\\
    |\lambda_q|\lesssim b_q^{[1]},
        &m<q\leq n_1
    \end{array}
\right\}.
\]
Moreover, \eqref{eq:app-first-s-new} becomes
\[
    \sum_{r=m_1+1}^m|\lambda_r|\sim1
\]
when $s_1$ is not the endpoint scale, and the left-hand side is
$\lesssim1$ at the endpoint scale.  Thus, at a non-endpoint scale, the largest index $r$ for which
$|\lambda_r|\sim1$ is greater than $m_1$; at the endpoint scale, it
may equal $m_1$.
Also, at the endpoint scale,
\[
    b_{n_1}^{[1]}\sim1.
\]
These are the data used at the next stage.

\subsection{The inductive step}

Suppose that stages $1,\ldots,i-1$ have been completed, where $i\geq2$,
and fix one of the intervals obtained at stage $i-1$.  In the original parameter variable, this interval has length comparable to
$\ell_{i-1}$, while after it is rescaled to $[0,1]$, the coefficient
range in order $q$ is $b_q^{[i-1]}$, where
\begin{equation}\label{eq:app-data-after-i-new}
    \ell_{i-1}
    :=
    \prod_{\nu<i}s_\nu,
    \qquad
    b_q^{[i-1]}
    :=
    \rho\prod_{\nu<i}s_\nu^{-(q-m_\nu)},
    \quad m<q\leq n_{i-1}.
\end{equation}
We again denote the rescaled curve by $\gamma$ and the coefficient
variables obtained from the preceding rescalings by $\lambda_r$.

First set
\begin{equation}\label{eq:app-next-n-new}
    n_i
    :=
    \begin{cases}
        n_{i-1},
        &\text{if $s_{i-1}$ was not the endpoint scale},\\
        n_{i-1}-1,
        &\text{if $s_{i-1}$ was the endpoint scale}.
    \end{cases}
\end{equation}
This is the update in \cite[(4.10)]{gan2024restricted}. If $n_i=m$, no
stage $i$ is performed, and stage $i-1$ is the last completed stage.
Henceforth assume that $n_i>m$.

In the second case of \eqref{eq:app-next-n-new}, the preceding rescaling
brings the coefficient range in order $n_{i-1}$ to unit size. Following
the alternative change of variables in
\cite[(4.19)]{gan2024restricted}, the ambient frequency coordinates
corresponding to orders $q>n_i$ are left unchanged in all subsequent
rescalings. Their ranges are bounded, so they are omitted from the
coefficient displays below and handled by Fubini.

After the first $i-1$ rescalings, decompose according to the largest
index $m_i\leq m$ whose coefficient has size comparable to one. On the
piece indexed by $m_i$,
\[
    |\lambda_{m_i}|\sim1,
    \qquad
    |\lambda_u|\ll1
    \quad(m_i<u\leq m).
\]
If $s_{i-1}\neq s_{i-1,\mathrm{end}}$, then
\cite[(4.11)--(4.12)]{gan2024restricted} gives
$m_i>m_{i-1}$.  If $s_{i-1}=s_{i-1,\mathrm{end}}$, then
\cite[(4.23)--(4.24)]{gan2024restricted} allows
$m_i=m_{i-1}$, so $m_i\geq m_{i-1}$.

Let $s_{i,\mathrm{end}}$ be the smallest dyadic number not smaller than
$\bigl(b_{n_i}^{[i-1]}\bigr)^{1/(n_i-m_i)}$; thus
\begin{equation}\label{eq:app-next-endpoint-new}
    \bigl(b_{n_i}^{[i-1]}\bigr)^{1/(n_i-m_i)}
    \leq s_{i,\mathrm{end}}
    <2\bigl(b_{n_i}^{[i-1]}\bigr)^{1/(n_i-m_i)}.
\end{equation}
For dyadic $s_{i,\mathrm{end}}\leq s_i\ll1$, divide the $m_i$-piece
according to
\begin{equation}\label{eq:app-next-s-new}
    \frac{|\lambda_{m_i+1}|}{s_i}
    +\cdots+
    \frac{|\lambda_m|}{s_i^{m-m_i}}
    \sim1,
\end{equation}
with $\lesssim1$ on the endpoint piece.  If $m_i=m$, the sum in
\eqref{eq:app-next-s-new} is empty.  In this case the $m_i$-piece is not
further decomposed in frequency according to $s_i$, and we set
$s_i=s_{i,\mathrm{end}}$.  

The choices of $m_i$ and $s_i$ are again made at each fixed parameter
value after the first $i-1$ rescalings.  The lower bounds selected at
the preceding stages are retained.  Together with the choice of
$m_i$, this gives
\[
    |\lambda_{m_\nu}|\sim1,
    \qquad 1\leq\nu\leq i.
\]
For each fixed parameter value $t$, the region selected by the
decompositions up to stage $i$ is contained in
\begin{equation}\label{eq:app-next-region-new}
\left\{
    \sum_{r=1}^{n_i}\lambda_r\gamma^{(r)}(t):
    \begin{array}{ll}
            |\lambda_r|\lesssim1,
        &r\in\{1,\ldots,m_i\}\setminus\{m_1,\ldots,m_i\},\\
    |\lambda_r|\sim1,
        &r\in\{m_1,\ldots,m_i\},\\
    |\lambda_r|\lesssim s_i^{r-m_i},
        &m_i<r\leq m,\\
    |\lambda_q|\lesssim b_q^{[i-1]},
        &m<q\leq n_i
    \end{array}
\right\}.
\end{equation}
For $m<q\leq n_i$, the definitions of $b_q^{[i-1]}$ and
$\ell_{i-1}$, together with the choice of $s_i$, give
\begin{equation}\label{eq:app-next-width-comparison-new}
    b_q^{[i-1]}
    =b_{n_i}^{[i-1]}\ell_{i-1}^{n_i-q}
    \leq b_{n_i}^{[i-1]}
    \leq s_i^{n_i-m_i}
    \leq s_i^{q-m_i}.
\end{equation}

Fix $(m_i,n_i,s_i)$.  Working in the parameter variable after the first
$i-1$ rescalings, divide $[0,1]$ into intervals $I$ of length comparable
to $s_i$ and choose a point $t_I\in I$.  After fixing the sign
of the $m_i$-th coefficient, fix $\sigma\in\{-1,1\}$ accordingly and
consider
\begin{equation}\label{eq:app-next-box-new}
\left\{
    \sum_{r=1}^{n_i}\mu_r\gamma^{(r)}(t_I):
    \begin{array}{ll}
    |\mu_r|\lesssim1,
        &1\leq r<m_i,\\
    \sigma\mu_{m_i}\sim1,
        &\\
    |\mu_r|\lesssim s_i^{r-m_i},
        &m_i<r\leq n_i
    \end{array}
\right\}.
\end{equation}
As in the first stage, Taylor expansion at $t_I$ shows that, for fixed
$m_i$, $n_i$, $s_i$, and $\sigma$, the union of the regions selected above as $t$ ranges over $I$ is contained in a fixed enlargement of
the box in \eqref{eq:app-next-box-new} based at $t_I$ and having sign
$\sigma$.  The comparison
\eqref{eq:app-next-width-comparison-new} shows that the ranges
$b_q^{[i-1]}$ satisfy the bounds in
\eqref{eq:app-next-box-new}. 

Within each fixed interval from the preceding stage, let $F_I$ have
Fourier support in a fixed enlargement of the box in
\eqref{eq:app-next-box-new} based at $t_I$ and having the fixed sign
$\sigma$.  The
$\ell^2$ form of the decoupling used in
\cite[Claim~4.1]{gan2024restricted}, applied in the derivative coordinates
of orders $m_i+1,\ldots,n_i$, gives
\begin{equation}\label{eq:app-inductive-decoupling-new}
\left\|
    \sum_I F_I
\right\|_p
\lesssim_\eta
\rho^{-\eta}
s_i^{-\left(
    \frac12-
    \frac{(n_i-m_i)(n_i+1-m_i)}{2p}
\right)}
\left(
    \sum_I\|F_I\|_p^2
\right)^{1/2}.
\end{equation}
After applying \eqref{eq:app-inductive-decoupling-new}, we continue the
argument separately on each interval $I$ and rescale the regions selected
by the decompositions up to stage $i$ as $t$ ranges over $I$. These
regions satisfy the bounds in \eqref{eq:app-next-region-new} together
with the condition \eqref{eq:app-next-s-new}.  Rescale each newly obtained interval $I$ affinely to $[0,1]$ and make
the accompanying change of frequency variables. Let $\widetilde\gamma$ denote the rescaled curve and let
$\lambda_r^{\mathrm{new}}$ denote the coefficient variable after this
rescaling. Up to fixed constants, $\lambda_r^{\mathrm{new}}$ is
$\lambda_r$ for $r\leq m_i$ and
$s_i^{-(r-m_i)}\lambda_r$ for $m_i<r\leq n_i$. In these new variables,
the terms of orders $m_i\leq r\leq n_i$ take the form
\[
    \sum_{r=m_i}^{n_i}
    \lambda_r^{\mathrm{new}}\widetilde\gamma^{(r)}(u),
\]
as in the first stage.  The directions of
orders $r<m_i$ are flat and are understood according to the convention
above, while the ambient frequency coordinates of orders $q>n_i$, which were
omitted from the displays above, remain unchanged.  Since $m_\nu\leq m_i$ for every $\nu\leq i$, the bounds
$|\lambda_{m_\nu}^{\mathrm{new}}|\sim1$ are preserved.  We then relabel $\widetilde\gamma$ as $\gamma$ and
$\lambda_r^{\mathrm{new}}$ as $\lambda_r$. Under this relabeling,
\eqref{eq:app-next-s-new} becomes
\[
    \sum_{r=m_i+1}^m|\lambda_r|\sim1,
\]
with $\lesssim1$ at the endpoint scale.  The preimage of $I$ in the
original parameter variable has length comparable to
$\ell_i=\ell_{i-1}s_i$, and the coefficient ranges become
\begin{equation}\label{eq:app-next-width-new}
    \ell_i=\ell_{i-1}s_i,
    \qquad
    b_q^{[i]}
    =
    b_q^{[i-1]}s_i^{-(q-m_i)},
    \qquad m<q\leq n_i.
\end{equation}
Thus the transformed regions have the form required to begin the next
stage.

\subsection{Outputs and iteration of decoupling}

At each transition before the procedure stops, the $m$-index increases or
the $n$-index decreases.  Hence the procedure stops after $J=O_n(1)$
stages, and the last completed stage has $n_J=m+1$.  The recorded data are
exactly an output in Proposition~\ref{prop:Vtheta-decomposition}:
\begin{equation}\label{eq:app-output-new}
    \bigl(m_i,n_i,s_i\bigr)_{i=1}^J,
    \qquad
    s_i=2^{-k_i},\quad k_i\in\mathbb N,
    \qquad
    \ell_{\bbk}=\prod_{i=1}^Js_i.
\end{equation}

Fix $q>m$.  No dyadic subdivision is made in the coefficient of order $q$; its
range changes only under the rescalings.  After
the last stage for which $n_i\geq q$, this range is
\[
    \rho\prod_{i:n_i\geq q}s_i^{-(q-m_i)}
\]
and is comparable to one.  Since each endpoint scale is the smallest dyadic
number above the exact threshold,
\begin{equation}\label{eq:app-product-new}
    \rho
    \leq
    \prod_{i:n_i\geq q}s_i^{q-m_i}
    <2^n\rho,
    \qquad m<q\leq n.
\end{equation}
This is the dyadic form of
\cite[(5.4), (5.7)]{gan2024restricted}.  Taking dyadic logarithms gives
\[
    0
    \leq
    j-\sum_{i:n_i\geq q}(q-m_i)k_i
    <n,
\]
which is \eqref{eq:output-identities}.  The case $q=m+1$ also gives
\[
    K_{\bbk}:=\sum_{i=1}^Jk_i\leq j,
    \qquad
    \ell_{\bbk}=2^{-K_{\bbk}}\geq2^{-j}.
\]
Together with $J=O_n(1)$, this shows that the number of outputs is
$O(j^{C_n})$.

We now recover the frequency decomposition. Undoing the rescaling at
stage $i$ contributes a factor $s_i^{r-m_i}$ to the coefficient range
of order $r$ when $m_i<r$, and a factor $1$ otherwise.
\begin{equation}\label{eq:app-tangent-width-new}
    L_r
    :=
    \prod_{i:m_i<r}s_i^{r-m_i},
    \qquad 1\leq r\leq m.
\end{equation}
These are the coefficient ranges produced by the decomposition in
orders $1,\ldots,m$. Since the coefficients of orders $q>m$ are not subdivided, undoing the
rescalings returns their ranges to width $\rho$.  After undoing the rescalings, the lower bounds recorded when the indices $m_i$ were selected, together
with the corresponding upper bounds, give
\[
    |\lambda_r|\sim L_r
    \quad(r\in\mathcal M_{\bbk}),
    \qquad
    |\lambda_r|\lesssim L_r
    \quad(1\leq r\leq m),
\]
where $\mathcal M_{\bbk}=\{m_1,\ldots,m_J\}$.

The choices of the scales $s_i$ ensure that
\eqref{eq:active-gap} holds for every
$r\in\mathcal M_{\bbk}$.  Lemma~\ref{lem:frame-comparison} therefore
allows us to replace these derivative-frame regions, up to fixed
enlargements and $O_n(1)$ subdivisions into sign components, by
comparable Frenet-frame output boxes.  Returning to the original frequency scale gives
\[
    2^jL_r
    =
    2^{\,j-\sum_{i:m_i<r}(r-m_i)k_i}
    =2^{a_r}
    \quad(1\leq r\leq m),
    \qquad
    2^j\rho=1.
\]
Since the construction was carried out separately for each parameter
value, this gives the smooth frequency decomposition in
Proposition~\ref{prop:Vtheta-decomposition}.

Fix an output $\bbk$ and functions $F_\tau$ satisfying
\[
    \operatorname{supp}\widehat{F_\tau}\subset C\tau,
    \qquad
    \tau\in\mathcal T^{\mathrm{net}}_{j,\bbk}.
\]
The only point to note is that, because of the fixed constant $c_0$ used
in Subsection~2.2, the parameter intervals used there need not agree
exactly with those arising in the iteration above.  The intervals obtained after the last stage of the iteration and the
output intervals of Subsection~2.2 both have length comparable to
$\ell_{\bbk}$. Consequently, each interval from either family meets only
$O_\gamma(1)$ intervals from the other family. By
Lemma~\ref{lem:nearby-output-boxes}, the boxes associated with
intersecting intervals are contained in constant enlargements of one
another.
Thus, at an $O_{\gamma,n}(1)$ cost, we may apply
\eqref{eq:app-one-stage-new} at the first stage and
\eqref{eq:app-inductive-decoupling-new} at stages $2,\ldots,J$ directly
to the family $\mathcal T^{\mathrm{net}}_{j,\bbk}$.

The Jacobian factors from the changes of variables occur on both sides
and cancel.  Since $J=O_n(1)$, the losses at the individual stages may
be chosen so that their product is at most $\rho^{-\eta}$.
Consequently,
\begin{equation}\label{eq:app-iterated-new}
\left\|
    \sum_{\tau\in\mathcal T^{\mathrm{net}}_{j,\bbk}}F_\tau
\right\|_p
\lesssim_\eta
\rho^{-\eta}
\prod_{i=1}^J
s_i^{-\left(
    \frac12-
    \frac{(n_i-m_i)(n_i+1-m_i)}{2p}
\right)}
\left(
    \sum_{\tau\in\mathcal T^{\mathrm{net}}_{j,\bbk}}
    \|F_\tau\|_p^2
\right)^{1/2}.
\end{equation}
Substituting $\rho=2^{-j}$ and $s_i=2^{-k_i}$ gives
\eqref{eq:output-decoupling}.

\bibliographystyle{alpha}
\bibliography{reference}

\end{document}